 \documentclass[11pt]{article}
 \usepackage{amsfonts,amssymb,mathrsfs}
 \usepackage{color}
 \usepackage[svgnames]{xcolor}
 \usepackage{bm}
 \usepackage{amsmath} 
\usepackage{amsthm}
\usepackage{xcolor}
\usepackage{hyperref}

 \allowdisplaybreaks
\makeatletter
\def\keywords{\xdef\@thefnmark{}\@footnotetext}
\makeatother
 \newtheorem{theorem}{Theorem}[section]
 \newtheorem{lemma}[theorem]{Lemma}
 \newtheorem{corollary}[theorem]{Corollary}
 \newtheorem{proposition}[theorem]{Proposition}
 \newtheorem{remark}[theorem]{Remark}
 \newtheorem{example}[theorem]{Example}
 \newtheorem{condition}[theorem]{Condition}
 \newtheorem{definition}[theorem]{Definition}

 \def\benumerate{\begin{enumerate}}\def\eenumerate{\end{enumerate}}
 \def\bitemize{\begin{itemize}}\def\eitemize{\end{itemize}}

 \def\beqlb{\begin{eqnarray}}
 \def\eeqlb{\end{eqnarray}}
 \def\beqnn{\begin{eqnarray*}}
 \def\eeqnn{\end{eqnarray*}}

 \def\qed{\hfill$\Box$\medskip}

 \def\<{\langle}\def\>{\rangle}

 \def\mrm{\mathrm}

 \def\ar{\!\!&}

 \def\d{\mrm{d}}
 \def\sgn{\mathrm{sgn}}\def\supp{\mathrm{supp}}

\begin{document}

\author{
Jieliang Hong$^\dagger$ \quad  Du Yang$^\S$  \quad  Junyan Zhang$^{*}$ 
}

\title{Boundary behavior and optimal pointwise H\"older exponent of the total local time of $(1+\beta)$-stable super-Brownian motion}

\date{{\small  {\it  
$^{\dagger}$$^\S$$^{*}$Department of Mathematics, Southern University of Science and Technology, \\Shenzhen, China\\
\quad \\
$^\dagger$E-mail:  {\tt hongjl@sustech.edu.cn} \\
    $^{\S}$E-mail:  {\tt yangd2026@mail.sustech.edu.cn} \\  
    $^{*}$E-mail:  {\tt zhangjunyan@mail.bnu.edu.cn} 
  }
  }
  }
 
 \maketitle
\keywords{ $^{*}$ Corresponding author: Junyan Zhang}
\keywords{This work was supported by the National Natural Science Foundation of China (Grant number: 12571150).}
\keywords{{\bf AMS 2020 subject classification:} 60J68; 60J80; 60G17}
\keywords{{\bf Key words and phrases:} stable super-Brownian motion; canonical measure; local time; pointwise H\"older exponent; Tanaka formula}
\begin{abstract}
Let $L^x$ be the total local time of one-dimensional super-Brownian motion with $(1+\beta)$-stable branching mechanism, $0<\beta<1$. We prove that $\{x\in\mathbb R:L^x>0\}$ is almost surely a bounded open interval $(\mathsf L,\mathsf R)$ and that
\[
h_L(\mathsf L)=h_L(\mathsf R)=1+\frac{2}{\beta}
\]
almost surely, where $h_L$ denotes the pointwise H\"older exponent. Thus the pointwise $\gamma$-H\"older condition holds at both endpoints for every $\gamma<1+2/\beta$ and fails for every $\gamma>1+2/\beta$. The same conclusions hold under the canonical excursion measure.
\end{abstract}


\section{Introduction}

\setcounter{equation}{0}

Superprocesses are measure-valued Markov processes that arise as scaling limits of
branching particle systems. Their laws are determined by a spatial motion and a
branching mechanism, and they provide canonical continuum models for populations
undergoing random migration and reproduction.  

In this paper we consider the one-dimensional super-Brownian motion with
$(1+\beta)$-stable branching mechanism, where $0<\beta<1$. In contrast with the
quadratic branching case, the martingale part is purely discontinuous and has stable
rather than Gaussian fluctuations (see, e.g., \cite{Da93} or Lemma 1.6 of \cite{FMW10}). This substantially changes the local regularity of
the process and makes the behavior of its occupation density near the edge of its
range more delicate.

Let $M_F=M_F(\mathbb R)$ be the space of finite measures on
$(\mathbb R,\mathscr B(\mathbb R))$, equipped with the topology of weak convergence,
and write
\[
 \mu(f)=\int_{\mathbb R}f(x)\mu(\d x),\qquad \mu\in M_F.
\]
Let $(\Omega,\mathscr F,(\mathscr F_t)_{t\geq0},\mathbb P)$ be a complete filtered
probability space. A super-Brownian motion $X=(X_t,t\geq0)$ with
$(1+\beta)$-stable branching mechanism is an $M_F$-valued strong Markov process
characterized by the Laplace functional
\begin{equation}\label{Laplace of X}
 \mathbb E_{X_0}\bigg[\exp\bigg\{-X_t(\phi)-\int_0^tX_s(f)\d s\bigg\}\bigg]
 =\exp\{-X_0(V_t(\phi,f))\},
\end{equation}
for nonnegative bounded Borel functions $\phi$ and $f$. Here
$V_t(x)=V_t(\phi,f)(x)$ is the unique nonnegative solution of
\begin{equation}\label{V_t}
 \frac{\partial V_t}{\partial t}=\frac12\Delta V_t-V_t^{1+\beta}+f,
 \qquad V_0=\phi.
\end{equation}
If $(P_t)_{t\geq0}$ denotes the Brownian semigroup, then equivalently
\begin{equation}\label{V related to P}
 V_t=P_t\phi+\int_0^tP_s\big(f-V_{t-s}^{1+\beta}\big)\d s.
\end{equation}
We denote by $\mathbb P_{X_0}$ and $\mathbb E_{X_0}$ the law and expectation of the
process started from $X_0\in M_F$.

The occupation measure of $X$ up to time $t$ is defined by
\[
 O_t(\cdot)=\int_0^tX_s(\cdot)\d s.
\]
Fleischmann \cite{Fl88} proved that, in the present one-dimensional setting, $O_t$
is absolutely continuous for every $t>0$. Its density, denoted by
$(L_t^x,t\geq0,x\in\mathbb R)$, is called the local time of $X$ and satisfies
\begin{equation}\label{occupation measure}
 \int_0^tX_s(\phi)\d s=\int_{\mathbb R}\phi(x)L_t^x\d x
\end{equation}
for every nonnegative measurable function $\phi$. Mytnik and Perkins \cite{MyP03}
showed that the local time has a jointly continuous version. If
\[
 \zeta=\inf\{t\geq0:X_t(1)=0\}
\]
denotes the extinction time, then $\zeta<\infty$ almost surely, and we write
\[
 L^x=L_\infty^x=L_\zeta^x
\]
for the total local time.

The local and global regularity of densities and occupation densities of stable
superprocesses has been studied by several authors; see, among others,
\cite{Fl88,FMW10,MyP03,MyX04,Xi05}. The question considered here is of a different
nature. We study the decay of $L^x$ at the random endpoints of its positivity set.
For ordinary one-dimensional super-Brownian motion, corresponding formally to
$\beta=1$, Hong \cite{Ho19} proved that the optimal boundary H\"older exponent is
$3$. Our main result shows that stable branching changes this exponent to
\[
 1+\frac{2}{\beta}.
\]
In particular, the quadratic value $3$ is recovered when $\beta=1$. The exponent is
considerably larger than the interior regularity exponent and reflects the rapid
vanishing of the occupation density at the boundary of its support.

The stable branching case differs from the quadratic case in two places. The martingale term is purely discontinuous and is treated by a time-change with a spectrally positive $(1+\beta)$-stable process. Under the canonical measure, the first variation of the occupation Laplace functional introduces the potential $\psi'(V)=(1+\beta)V^\beta$, leading to a localized Feynman--Kac argument.

For the positivity set, the exit mass $Y_r\delta_r$ from $(-\infty,r)$ is identified as a $(1+\beta/2)$-stable continuous-state branching process. Its extinction point gives the right endpoint, while the explicit Laplace transform of $L^x$, a deterministic grid, and Borel--Cantelli yield strict positivity before extinction.

For the upper estimate, the Tanaka formula and the stable time-change improve a decay exponent $\xi_0$ to any
\[
 \xi_1<1+\frac{1+\xi_0}{1+\beta}.
\]
Iteration converges to the fixed point $1+2/\beta$. The matching lower estimate follows from the extinction behavior of the same spatial continuous-state branching process.

Let
$C_b^2(\mathbb R)$ denote the space of twice continuously differentiable bounded
functions whose first two derivatives are bounded. The process $X$ is the unique in
law solution of the following martingale problem: for every
$\phi\in C_b^2(\mathbb R)$,
\begin{equation}\label{martingale problem}
 M_t(\phi)=X_t(\phi)-X_0(\phi)-\int_0^tX_s\bigg(\frac{\Delta\phi}{2}\bigg)\d s
\end{equation}
is an $(\mathscr F_t)$-martingale; see Dawson \cite{Da93}. In the stable branching
case this martingale is purely discontinuous; see Lemma 1.6 of \cite{FMW10}.
Let
\[
 p_s(x)=(2\pi s)^{-1/2}\exp\bigg(-\frac{x^2}{2s}\bigg)
\]
be the one-dimensional Brownian transition density and, for $\lambda>0$, set
\[
 G_\lambda(x)=\int_0^\infty e^{-\lambda s}p_s(x)\d s
 =\frac{1}{\sqrt{2\lambda}}e^{-\sqrt{2\lambda}|x|},
 \qquad G_\lambda^x(y)=G_\lambda(y-x).
\]
The Tanaka formula of Mytnik and Xiang \cite{MyX04} and Xiang \cite{Xi05} gives
\begin{equation}\label{Tanaka formula for local time}
 L_t^x=X_0(G_\lambda^x)-X_t(G_\lambda^x)
 +\lambda\int_0^tX_s(G_\lambda^x)\d s+M_t(G_\lambda^x).
\end{equation}

We also use the canonical measure $\mathbb N_{x_0}$ and the L\'evy-snake representation. If $\Xi$ is a Poisson point process with intensity
\[
 \mathbb N_{X_0}=\int_{\mathbb R}\mathbb N_{x_0}(\cdot)X_0(\d x_0),
\]
then, under $\mathbb P_{X_0}$,
\begin{equation}\label{X poisson point representation 1}
 X_t=\int\nu_t\,\Xi(\d\nu),
\end{equation}
and the local time decomposes accordingly as
\begin{equation}\label{L poisson point representation 1}
 L_t^x=\int L_t^x(\nu)\,\Xi(\d\nu).
\end{equation}
We use the same notation $\mathbb N_{x_0}$ for the canonical measure of the
superprocess and for the corresponding excursion measure of the L\'evy snake; see
Theorem 4.2.1 of \cite{DuL02}.

Finally, let $H$ be a bounded regular open interval such that
$\supp(X_0)\subset H$ and assume
\[
 d(\supp(X_0),H^c)>0.
\]
Let $X_H$ denote the exit measure from $H$. For every nonnegative continuous
function $h$ on $\partial H$,
\begin{equation}\label{exit Laplace H}
 \mathbb E_{X_0}\big[e^{-X_H(h)}\big]
 =\exp\bigg\{-\int_{\mathbb R}U^h(x)X_0(\d x)\bigg\},
\end{equation}
where $U^h$ is the unique nonnegative solution of
\begin{equation}\label{exit equation Ug}
 \begin{aligned}
  \frac12\Delta U^h&=(U^h)^{1+\beta} &&\text{in }H,\\
  U^h&=h &&\text{on }\partial H.
 \end{aligned}
\end{equation}

We now state the main results. Throughout the paper the process starts from
$\delta_0$ unless otherwise indicated.

\begin{theorem}\label{random interval}
 There exist finite random variables $\mathsf L<0<\mathsf R$ such that
 \[
  \{x\in\mathbb R:L^x>0\}=(\mathsf L,\mathsf R),
  \qquad \mathbb P_{\delta_0}\text{-a.s.}
 \]
\end{theorem}

The explicit Laplace transform used in the proof also yields the following
distributional identity, which extends Corollary 6 of \cite{LeR20}.

\begin{corollary}\label{stable rv coro}
 The random variable $L^0$ has a positive stable distribution of index
 $2/(2+\beta)$.
\end{corollary}

\begin{definition}\label{pointwise Holder definition}
 Let $\gamma>0$. We say that a function $f:\mathbb R\to\mathbb R$ satisfies the
 \emph{pointwise $\gamma$-H\"older condition} at $x\in\mathbb R$ if there are
 constants $\delta,c>0$ such that
 \[
  |f(y)-f(x)|\leq c|y-x|^\gamma
 \]
 whenever $|y-x|<\delta$. Its pointwise H\"older exponent at $x$ is
 \[
  h_f(x):=\sup\bigl\{\gamma>0:
  f\text{ satisfies the pointwise $\gamma$-H\"older condition at }x\bigr\}.
 \] 
\end{definition}

Throughout, the pointwise H\"older exponent is understood in the sense of
Definition \ref{pointwise Holder definition}. At the boundary, this exponent measures the order of vanishing of \(L^x\).

\begin{theorem}\label{upper Holder continuity}
 If $0<\gamma<1+2/\beta$, then, $\mathbb P_{\delta_0}$-almost surely, the map
 $x\mapsto L^x$ satisfies the pointwise $\gamma$-H\"older condition at both $\mathsf L$ and
 $\mathsf R$.
\end{theorem}

\begin{theorem}\label{lower Holder continuity}
 If $\gamma>1+2/\beta$, then, $\mathbb P_{\delta_0}$-almost surely, the map
 $x\mapsto L^x$ fails the pointwise $\gamma$-H\"older condition at each of
 $\mathsf L$ and $\mathsf R$.
\end{theorem}

Taking countable intersections over rational exponents in Theorems
\ref{upper Holder continuity} and \ref{lower Holder continuity} gives
\[
 h_L(\mathsf L)=h_L(\mathsf R)=1+\frac{2}{\beta},
 \qquad \mathbb P_{\delta_0}\text{-a.s.}
\]

The corresponding assertions also hold for a single cluster under the canonical
measure. 

\begin{theorem}\label{canonical random interval}
 Under $\mathbb N_0$, there exist finite random variables
 $\mathsf L<0<\mathsf R$ such that
 \[
  \{x\in\mathbb R:L^x>0\}=(\mathsf L,\mathsf R),
  \qquad \mathbb N_0\text{-a.e.}
 \]
\end{theorem}

\begin{theorem}\label{canonical Holder continuity}
 Theorems \ref{upper Holder continuity} and
 \ref{lower Holder continuity} remain valid when
 $\mathbb P_{\delta_0}$ is replaced by $\mathbb N_0$. More precisely,
 $\mathbb N_0$-a.e., $x\mapsto L^x$ satisfies the pointwise $\gamma$-H\"older condition
 at both $\mathsf L$ and $\mathsf R$ for every $0<\gamma<1+2/\beta$, and it fails
 this condition at each endpoint for every $\gamma>1+2/\beta$. Equivalently,
 \[
 h_L(\mathsf L)=h_L(\mathsf R)=1+\frac{2}{\beta},
 \qquad \mathbb N_0\text{-a.e.}
 \]
\end{theorem}

There is also a canonical-measure version of Corollary
\ref{stable rv coro}.  

\begin{corollary}\label{canonical stable Levy measure}
 Put
 \[
  \alpha=\frac{2}{2+\beta},
  \qquad
  c_\beta=\left(\frac{\sqrt{\beta+2}}{2}\right)^\alpha.
 \]
 Then
 \[
  \mathbb N_0\big(1-e^{-\lambda L^0}\big)
  =c_\beta\lambda^\alpha,
  \qquad \lambda>0,
 \]
 and  
\begin{align}\label{e3.7}
  \mathbb N_0(L^0\in\d \ell)
  =\frac{c_\beta\alpha}{\Gamma(1-\alpha)}
    \ell^{-1-\alpha}\d\ell,
  \qquad \ell>0.
\end{align}
 In particular, $L^0>0$, $\mathbb N_0$-a.e.
\end{corollary}

The paper is organized as follows. Section 2 recalls the
L\'evy-snake construction and exit measures, derives the explicit Laplace transform
of the total local time, which implies Corollary \ref{stable rv coro}, 
and collects the preliminary continuity estimates. In
Section 3 we identify the half-line exit-measure process as a stable continuous-state
branching process and prove Theorem \ref{random interval}. Section 4 establishes the Tanaka-type identities needed near
the boundary and proves the upper estimate by an iterative dyadic argument. Section
5 proves the matching lower estimate from the extinction behavior of the spatial
continuous-state branching process, and hence completes the proof of Theorem
\ref{lower Holder continuity}. Section 6 transfers all the main conclusions to the
canonical measure by a Poisson cluster argument. The only additional point is the
strict positivity of $L^0$ under $\mathbb N_0$, which is obtained from a first-variation
formula for the occupation Laplace functional.

\section{Preliminaries}

\setcounter{equation}{0}

We recall the canonical measure and exit measures, derive the one-point Laplace transform of the total local time, and record the spatial continuity estimate used below.

\subsection{Canonical measure and exit measures}

We recall only those features of the L\'evy-snake construction that enter the
arguments below; see Chapter 4 of \cite{DuL02} for the general construction.
Let $\mathbb D(I,E)$ be the Skorokhod space of c\`adl\`ag paths from
$I\subset\mathbb R_+$ to $E$, and set
\[
 \mathcal W=\bigcup_{t\geq0}\mathbb D([0,t],\mathbb R).
\]
For $\omega\in\mathbb D([0,t],\mathbb R)$, write $\zeta_\omega=t$ for its
lifetime and $\widehat\omega=\omega(\zeta_\omega)$ for its terminal point.
For stable branching, the strong Markov L\'evy-snake state is the pair
$(\rho_s,W_s)_{s\geq0}$, where $\rho$ is the exploration process and $W$ the
spatial path component; see Chapter~4 of \cite{DuL02}.  We write $W$ for the
corresponding excursion when no confusion can arise.  Its excursion measure
$\mathbb N_{x_0}$ is identified with the canonical measure of the
$(1+\beta)$-stable super-Brownian motion started from $x_0$.

The feature of this representation that we use most often is its Poissonian
decomposition.  If
\[
 \Xi=\sum_{i\in I}\delta_{W_i}
\]
is a Poisson point measure of L\'evy-snake excursions with intensity
$\mathbb N_{X_0}=\int_{\mathbb R}\mathbb N_{x_0}(\cdot)X_0(\d x_0)$, then
\begin{equation}\label{X poisson point representation}
 X_t=\sum_{i\in I}X_t(W_i)
     =\int X_t(W)\,\Xi(\d W),\qquad t>0,
\end{equation}
has law $\mathbb P_{X_0}$.  The occupation density decomposes cluster by
cluster:
\begin{equation}\label{L poisson point representation}
 L_t^x=\sum_{i\in I}L_t^x(W_i)
      =\int L_t^x(W)\,\Xi(\d W).
\end{equation}
If $\zeta_s=\zeta_{W_s}$ denotes the lifetime process of the snake, we write
\[
 \sigma=\inf\{s>0:\zeta_s=0\}
\]
for the duration of an excursion.
For every nonnegative measurable function $\phi$, the total occupation measure of
one excursion satisfies
\begin{equation}\label{canonical occupation identity}
 \mathcal O(\phi)
 :=\int_0^\infty X_t(\phi)\d t
 =\int_{\mathbb R}\phi(x)L^x\d x
 =\int_0^\sigma\phi(\widehat W_s)\d s,
 \qquad \mathbb N_{x_0}\text{-a.e.}
\end{equation}
In particular, $\mathcal O(1)=\sigma\in(0,\infty)$,
$\mathbb N_{x_0}$-a.e.

Exit measures connect the spatial branching structure to nonlinear elliptic
equations.  Let $D\subset\mathbb R$ be open; the exit measure is defined for
general open sets in \cite[Section~4.3]{DuL02}.  For a killed path
$\omega\in\mathcal W$, define
\[
 \tau_D(\omega)=\inf\{r\in[0,\zeta_\omega):\omega(r)\notin D\},
 \qquad \inf\varnothing=\infty.
\]
The exit local time from $D$ is
\[
 L_s^D=\lim_{\varepsilon\downarrow0}\frac1\varepsilon
 \int_0^s
 \mathbf 1_{\{\tau_D(W_r)<\zeta_r<\tau_D(W_r)+\varepsilon\}}\d r,
\]
and the associated exit measure is
\[
 \langle\mathcal Z^D,g\rangle
 =\int_0^\sigma g(\widehat W_s)\,\d L_s^D.
\]
For a bounded regular domain $D$ and a nonnegative continuous function
$g$ on $\partial D$,
\begin{equation}\label{exit canonical Laplace}
 u(x)=\mathbb N_x\big(1-e^{-\langle\mathcal Z^D,g\rangle}\big)
\end{equation}
is the unique nonnegative solution of
\begin{equation}\label{exit equation}
 \begin{cases}
  \frac12\Delta u=u^{1+\beta},&x\in D,\\
  u=g,&x\in\partial D.
 \end{cases}
\end{equation}

We use the following stability lemma.

\begin{lemma}\label{pointwise closed}
 Let $D$ be a bounded open interval.  The class of nonnegative
 $C^2$-solutions of
 \[
  \Delta u=2u^{1+\beta}\qquad\text{in }D
 \]
 is closed under finite pointwise convergence.
\end{lemma}
\begin{proof}
 Let $U\Subset D$ be a bounded regular subinterval.  Every nonnegative
 solution $v$ on $D$ satisfies, by \eqref{exit canonical Laplace}--\eqref{exit equation},
 \[
  v(x)=\mathbb N_x\big(1-e^{-\langle\mathcal Z^U,v\rangle}\big),
  \qquad x\in U.
 \]
 Let $(u_n)$ be nonnegative solutions on $D$ such that
 $u_n(x)\to u(x)<\infty$ for every $x\in D$.  Thus
 \begin{equation}\label{u_n}
  u_n(x)=\mathbb N_x\big(1-e^{-\langle\mathcal Z^U,u_n\rangle}\big),
  \qquad x\in U.
 \end{equation}
 Put $r=\frac12\operatorname{dist}(\overline U,D^c)>0$.  For every
 $y\in\partial U$, the ball $B(y,r)$ is contained in $D$, and the exit-measure
 representation gives
 \[
  u_n(y)\leq\mathbb N_y(\mathcal Z^{B(y,r)}\neq0)
  =\mathbb N_0(\mathcal Z^{B(0,r)}\neq0).
 \]
 The last quantity is finite by Theorem 4.5.2 of \cite{DuL02}, since
 \[
  \int_1^\infty
  \left(\int_0^t s^{1+\beta}\d s\right)^{-1/2}\d t<\infty.
 \]
 Hence $(u_n)$ is uniformly bounded on $\partial U$.  Since
 $\mathcal Z^U$ is a finite random measure supported on $\partial U$,
 pointwise convergence and the uniform bound imply
 \[
  \langle\mathcal Z^U,u_n\rangle
  \longrightarrow
  \langle\mathcal Z^U,u\rangle
 \]
 $\mathbb N_x$-a.e. on $\{\mathcal Z^U\neq0\}$.  Moreover,
 $1-e^{-\langle\mathcal Z^U,u_n\rangle}\leq
 \mathbf1_{\{\mathcal Z^U\neq0\}}$, whose $\mathbb N_x$-mass is finite.
 Dominated convergence in \eqref{u_n} therefore yields
 \[
  u(x)=\mathbb N_x\big(1-e^{-\langle\mathcal Z^U,u\rangle}\big),
  \qquad x\in U.
 \]
 By \eqref{exit equation}, $u$ is a $C^2$-solution on $U$.  Since
 $U\subset D$ was arbitrary, the conclusion follows. \qed
\end{proof}

\subsection{The Laplace transform of the total local time}

The next two lemmas give the one-point Laplace transform of $L^x$ and its scaling.

\begin{lemma}\label{lemma V lambda}
	For any $X_0\in M_F(\mathbb R)$ and $\lambda>0$,
	\begin{equation}\label{Laplace of V lambda}
		\mathbb E_{X_0}(\exp(-\lambda L^x))=\exp\bigg(-\int V^\lambda(x-x_0)X_0(\d x_0) \bigg),
	\end{equation}
	where $V^\lambda$ is the unique bounded nonnegative distributional solution of
	\begin{equation}\label{V lambda}
		\frac{\Delta V^\lambda}{2}=(V^\lambda)^{1+\beta}-\lambda\delta_0
	\end{equation}
	such that $V^\lambda\in C(\mathbb R)\cap C^2(\mathbb R\setminus\{0\})$ and
	$V^\lambda(x)\to0$ as $|x|\to\infty$.
\end{lemma}

\begin{proof}
	\textbf {Step 1.} Let $\{r_\varepsilon:\varepsilon\in(0,1)\}$ be a smooth, radially symmetric approximate identity satisfying $\{r_\varepsilon>0\}\subset B(0,\sqrt{\varepsilon})$, $r_\varepsilon\leq c_0p_\varepsilon$ and $r_\varepsilon\to\delta_0$ as $\varepsilon \to 0$ (here the convergence is in the distributional sense). Set $r_\varepsilon^x(y)=r_\varepsilon(y-x)$. Applying the Laplace functional \eqref{Laplace of X} with $\phi=0$ and $f=\lambda r_\varepsilon^x$, we have
	\begin{equation}\label{Laplace 2.2}
	\mathbb E_{X_0}(\exp\{-\lambda O_t(r_\varepsilon^x) \})=\exp\{-X_0(V_t(0,\lambda r_\varepsilon^x))\},
	\end{equation}
	where 
	$V_t(0,\lambda r_\varepsilon^x)$ is the unique solution of
	\begin{equation}\label{V_t mollified}
		\frac{\partial V_t}{\partial t}=\frac{\Delta V_t}{2}-V_t^{1+\beta}+\lambda r_\varepsilon^x, \qquad V_0=0.
	\end{equation}	
	Then by Theorem 3.3 of \cite{Is86} and symmetry, we have
	$$
	V_t(0,\lambda r_\varepsilon^x)(x_0)=V_t(0,\lambda r_\varepsilon)(x-x_0)\uparrow V^{\lambda,\varepsilon}(x-x_0)\ \text{as}\ t\to \infty
	$$
	where $V^{\lambda,\varepsilon}$ is the unique solution of
	\begin{equation}\label{V lambda varepsilon}
		\frac{\Delta V^{\lambda,\varepsilon}}{2}=(V^{\lambda,\varepsilon})^{1+\beta}-\lambda r_\varepsilon,\quad V^{\lambda,\varepsilon}>0\ \text{on}\ \mathbb R.
	\end{equation}
	Letting $t\to\infty$ in \eqref{Laplace 2.2}, together with \eqref{occupation measure}, we obtain 
	\begin{equation}\label{the limit of Laplace 2.2}
		\mathbb E_{X_0}\bigg(\exp\bigg\{-\lambda\int_{\mathbb R} r_\varepsilon^x(y)L^y\d y\bigg\}\bigg)=\mathbb E_{X_0}(\exp\{-\lambda O_\infty(r_\varepsilon^x) \})=\exp\{-X_0(V^{\lambda,\varepsilon}(x-\cdot))\}.
	\end{equation}
	
	\textbf{Step 2.} Let $X_0=\delta_{x_0}$. Then \eqref{the limit of Laplace 2.2} reduces to
	\begin{equation}\label{the limit of Laplace 2.2 special case}
	\mathbb E_{\delta_{x_0}}\bigg(\exp\bigg\{-\lambda\int_{\mathbb R} r_\varepsilon^x(y)L^y\d y\bigg\}\bigg)=\exp\{-V^{\lambda,\varepsilon}(x-x_0)\}.
	\end{equation}
	For $x\neq x_0$, continuity of $L^\cdot$ and $r_\varepsilon^x\to\delta_x$ imply that the left-hand side of \eqref{the limit of Laplace 2.2 special case} converges to $\mathbb E_{\delta_{x_0}}(\exp\{-\lambda L^x\})\in (0,1)$ as $\varepsilon\to 0$. Thus, for every $z\neq0$ (take $z=x-x_0$), $V^{\lambda,\varepsilon}(z)$ converges to a finite limit $V^{\lambda,0}(z)\in(0,\infty)$. Hence
	\begin{equation}\label{Lx for x neq x0}
		\mathbb E_{\delta_{x_0}}(\exp\{-\lambda L^x\})=\exp\{-V^{\lambda,0}(x-x_0)\}\ \text{for}\ x\neq x_0.
	\end{equation}
	
	For $|x|>\sqrt{\varepsilon}$, \eqref{V lambda varepsilon} gives $\Delta V^{\lambda,\varepsilon}=2(V^{\lambda,\varepsilon})^{1+\beta}$. Lemma \ref{pointwise closed} implies that $V^{\lambda,0}$ is a nonnegative $C^2$ solution of $\Delta V^{\lambda,0}=2(V^{\lambda,0})^{1+\beta}$ on $\mathbb R\backslash\{0\}$.
	We claim that
	\begin{equation}\label{estimation of V lambda varepsilon}
		V^{\lambda,\varepsilon} (x)\leq c_0\lambda+\beta^{-1/\beta},\quad \forall x\in\mathbb R,\ \varepsilon\in(0,1).
	\end{equation}
	In fact, the semigroup property implies that for any $t>1$,
	\begin{eqnarray*}
		V_t(0,\lambda r_\varepsilon^x)\ar=\ar V_1(V_{t-1}(0,\lambda r_\varepsilon^x),\lambda r_\varepsilon^x)\leq  \lim_{n\to\infty} V_1(n,\lambda r_\varepsilon^x)\\
		\ar\leq \ar\lim_{n\to\infty} V_1(n,0)+V_1(0,\lambda r_\varepsilon^x),
	\end{eqnarray*}
	where we use the boundedness of $V_{t-1}(0,\lambda r_\varepsilon^x)$ and monotonicity of $V_1$ in the initial condition for the first inequality; the second follows from subadditivity with respect to the initial condition and source term. Indeed, $(u+v)^{1+\beta}\geq u^{1+\beta}+v^{1+\beta}$ and the comparison principle yield the stated bound; see, e.g., \cite{Fr64}. By elementary calculus, we have $V_1(n,0)\equiv(n^{-\beta}+\beta)^{-1/\beta}\leq \beta^{-1/\beta}$. By \eqref{V related to P}, we have
	$$
	V_1(0,\lambda r_\varepsilon^x)(y)\leq \int_0^1 P_s(\lambda r_\varepsilon^x)(y)\d s\leq \lambda c_0\int_0^1 p_{s+\varepsilon}(y-x)\d s\leq \lambda c_0, \quad\forall\varepsilon\in(0,1).
	$$
	Hence \eqref{estimation of V lambda varepsilon} holds. The uniform bound and pointwise convergence imply, by dominated convergence on compact sets, that $V^{\lambda,\varepsilon}\to V^{\lambda,0}$ and $(V^{\lambda,\varepsilon})^{1+\beta}\to (V^{\lambda,0})^{1+\beta}$ in $L^1_{\mathrm{loc}}(\mathbb R)$. Hence we may pass to the limit in the distributional form of \eqref{V lambda varepsilon}, and $V^{\lambda,0}$ satisfies \eqref{V lambda}. Taking $x_0=0$ in \eqref{Lx for x neq x0}, the compact support property up to the almost surely finite extinction time gives $L^x\to0$ as $|x|\to\infty$, almost surely. Dominated convergence then yields $V^{\lambda,0}(x)\to0$ as $|x|\to\infty$. Lemma \ref{explicit V lambda} below shows that this solution is the unique bounded nonnegative distributional solution vanishing at infinity; we therefore write $V^\lambda=V^{\lambda,0}$. Since a one-dimensional distributional solution of \eqref{V lambda} in this class has a continuous representative, we may let $x\to x_0$ on both sides of \eqref{Lx for x neq x0}. Thus 
	\begin{equation}\label{Lx for all x}
		\mathbb E_{\delta_{x_0}}(\exp\{-\lambda L^x\})=\exp\{-V^{\lambda}(x-x_0)\}
	\end{equation}
	holds for all $x$ by the continuity of $V^\lambda$ at 0 and $L^x$ at $x_0$.
	
	\textbf{Step 3.} Let $X_0\in M_F(\mathbb R)$. By \eqref{L poisson point representation}, we have
	\begin{equation}\label{general Laplace of Lx}
		\mathbb E_{X_0}(\exp \{-\lambda L^x\})=\exp\bigg\{-\int\int (1-e^{-\lambda L^x(\nu)})\mathbb N_{x_0}(\d \nu)X_0(\d x_0)\bigg\}.
	\end{equation}
	Taking $X_0=\delta_{x_0}$ both in \eqref{Lx for all x} and \eqref{general Laplace of Lx} gives
	$$
	V^\lambda(x-x_0)=\int (1-e^{-\lambda L^x})\d \mathbb N_{x_0}.
	$$
	Substituting the above into \eqref{general Laplace of Lx}, we derive \eqref{Laplace of V lambda} under $\mathbb P_{X_0}$.\qed
\end{proof}

The solution $V^\lambda$ is explicit.

\begin{lemma}\label{explicit V lambda}
	Let $V^\lambda$ be the bounded nonnegative distributional solution to \eqref{V lambda} which vanishes at infinity. Then we have
	\begin{equation}\label{precise V lambda}
		V^\lambda(x)=\bigg[\frac{\beta}{\sqrt{\beta+2}}|x|+\bigg(\frac{2}{\lambda\sqrt{\beta+2}}\bigg)^{\frac{\beta}{2+\beta}}\bigg]^{-\frac{2}{\beta}},\quad x\in\mathbb R,\ \lambda>0.
	\end{equation}
\end{lemma}
\begin{proof}
	Let $u$ be the right-hand side of \eqref{precise V lambda}. We first verify \eqref{V lambda} in the distributional sense. Let $C_c^\infty(\mathbb R)$ be the space of smooth compactly supported functions and fix $\lambda>0$. By elementary calculus, we have $u(x)=u(-x)$ and for $x>0$,
	\begin{eqnarray*}
		u'(x)=-\frac{2}{\sqrt{\beta+2}}u(x)^{1+\beta/2},\quad	u''(x)=2u(x)^{1+\beta},\quad u(0)=\bigg(\frac{2}{\lambda\sqrt{\beta+2}}\bigg)^{-\frac{2}{2+\beta}}.
	\end{eqnarray*}
	For any $f\in C_c^\infty(\mathbb R)$,
	\begin{align}\label{deltauf}
		\frac{\Delta}{2}\int_{\mathbb R}u(x-y)f(y)\d y&=\frac{1}{2}\int_{\mathbb R}u(x-y)f''(y)\d y \notag \\
		&=\frac{1}{2}\lim_{\varepsilon\downarrow 0}\int_{x+\varepsilon}^\infty u(x-y)f''(y)\d y+\frac{1}{2}\lim_{\varepsilon\downarrow 0}\int_{-\infty}^{x-\varepsilon} u(x-y)f''(y)\d y.
	\end{align}
	For any $\varepsilon>0$, by the property of $u$ and integration by parts, we have
	\begin{eqnarray*}
		\int_{x+\varepsilon}^\infty u(x-y)f''(y)\d y\ar=\ar\int_{x+\varepsilon}^\infty u(y-x)f''(y)\d y=\int_{\varepsilon}^\infty u(z)f''(x+z)\d z\\
		\ar=\ar\Big[u(z)f'(x+z)\Big]_\varepsilon^\infty-\int_\varepsilon^\infty u'(z)f'(x+z)\d z\\
		\ar=\ar-u(\varepsilon)f'(x+\varepsilon)-\Big[u'(z)f(x+z)\Big]_\varepsilon^\infty+\int_\varepsilon^\infty u''(z)f(x+z)\d z\\
		\ar=\ar-u(\varepsilon)f'(x+\varepsilon)+u'(\varepsilon)f(x+\varepsilon)+2\int_\varepsilon^\infty u(z)^{1+\beta}f(x+z)\d z\\
		\ar=\ar-u(\varepsilon)f'(x+\varepsilon)+u'(\varepsilon)f(x+\varepsilon)+2\int_{x+\varepsilon}^\infty u(x-y)^{1+\beta}f(y)\d y
	\end{eqnarray*}
	and
	\begin{eqnarray*}
		\int_{-\infty}^{x-\varepsilon} u(x-y)f''(y)\d y\ar=\ar\int_\varepsilon^\infty u(z)f''(x-z)\d z\\
		\ar=\ar\Big[-u(z)f'(x-z)\Big]_\varepsilon^\infty+\int_\varepsilon^\infty u'(z)f'(x-z)\d z\\
		\ar=\ar u(\varepsilon)f'(x-\varepsilon)+\Big[-u'(z)f(x-z)\Big]_\varepsilon^\infty+\int_\varepsilon^\infty u''(z)f(x-z)\d z\\
		\ar=\ar u(\varepsilon)f'(x-\varepsilon)+u'(\varepsilon)f(x-\varepsilon)+2\int_\varepsilon^\infty u(z)^{1+\beta}f(x-z)\d z\\
		\ar=\ar u(\varepsilon)f'(x-\varepsilon)+u'(\varepsilon)f(x-\varepsilon)+2\int_{-\infty}^{x-\varepsilon} u(x-y)^{1+\beta}f(y)\d y.
	\end{eqnarray*}
	Let $\varepsilon\downarrow 0$ in above two equations. Since $u'(0+)=-\lambda$, the continuity of $u$ at 0 and $f\in C_c^\infty(\mathbb R)$, now \eqref{deltauf} becomes
	\begin{eqnarray*}
		\frac{\Delta}{2}\int_{\mathbb R}u(x-y)f(y)\d y\ar=\ar -\lambda f(x)+\int_{\mathbb R} u(x-y)^{1+\beta}f(y)\d y\\
		\ar=\ar\int_{\mathbb R} u(x-y)^{1+\beta}f(y)\d y-\lambda\int_{\mathbb R} \delta_0(x-y)f(y)\d y,
	\end{eqnarray*}
	implying that the function defined by \eqref{precise V lambda} is a solution to \eqref{V lambda} in the distributional sense.

	It remains to verify uniqueness in the stated class. Let $v$ be any bounded nonnegative distributional solution which vanishes at infinity. Then $v$ is continuous on $\mathbb R$, is $C^2$ on $\mathbb R\setminus\{0\}$, and
	\[
	v'(0+)-v'(0-)=-2\lambda.
	\]
	On $(0,\infty)$, $v''=2v^{1+\beta}\geq0$, so $v'$ is increasing. Since $v\geq0$ and $v(x)\to0$ as $x\to\infty$, necessarily $v'(x)\uparrow0$; otherwise $v$ could not converge to zero. The same argument, after reflection, applies on $(-\infty,0)$. Hence, on each half-line, multiplying $v''=2v^{1+\beta}$ by $v'$ and using $v,v'\to0$ at the corresponding infinity gives
	\[
	(v')^2=\frac{4}{\beta+2}v^{\beta+2}.
	\]
	Consequently $v'(x)=-2(\beta+2)^{-1/2}v(x)^{1+\beta/2}$ for $x>0$, while the opposite sign holds for $x<0$. The jump condition determines $v(0)$ uniquely, and integration of these first-order equations yields exactly \eqref{precise V lambda}. This proves uniqueness. \qed
\end{proof}

\medskip
The explicit expression immediately gives the scaling identity
\begin{equation}\label{V lambda scaling}
 V^\lambda(x)
 =r^{-2/\beta}
 V^{\lambda r^{1+2/\beta}}\left(\frac{x}{r}\right),
 \qquad r>0.
\end{equation}
This relation will be used in the dyadic positivity argument in Section 3.

 \noindent\textbf{Proof of Corollary \ref{stable rv coro}.}
	Plugging $x=0$ into \eqref{precise V lambda} gives
	\begin{equation*}
		V^\lambda(0)=\bigg(\frac{\lambda\sqrt{\beta+2}}{2}\bigg)^{\frac{2}{2+\beta}},\ \lambda>0.
	\end{equation*}
	Then by \eqref{Laplace of V lambda} with $X_0=\delta_0$, we obtain the desired result. \qed
	
\subsection{Some spatial continuity estimates}

The boundary exponent obtained in this paper is much larger than the ordinary
interior H\"older exponent. Nevertheless, a weak spatial estimate is needed at
two points in the proof: it allows us to extend lower bounds from deterministic
grids to all spatial points in Section 3, and it supplies the initial exponent for
the bootstrap in Section 4. We record the required localized form below.

We use the following quantitative form of Kolmogorov's continuity criterion.

\begin{lemma}\label{Kolmogorov's continuity criterion}
	Let $\{K_t : t \in [0,T]\}$ be a c\`adl\`ag stochastic process. Suppose there are constants $a > 1$, $b > 0$, and $C > 0$ such that for all $s,t \in [0,T]$,
	\[
	\mathbb{E}|K_t - K_s|^a \le C |t-s|^{1+b}.
	\]
	Fix $\gamma \in (0, b/a)$. Define
	\[
	\bar\rho_\gamma := \sup\{ r\in[0,T] : \forall s,t \in [0,T],\ |t-s| \leq r \Rightarrow |K_t - K_s| \leq |t-s|^\gamma \},
	\qquad \rho_\gamma:=\frac12\bar\rho_\gamma.
	\]
	Then 
	\begin{itemize}
		\item[{\rm(1)}] $|K_t - K_s| \leq |t-s|^\gamma$ whenever $|t-s| \leq\rho_\gamma$,\ $s,t \in [0,T]$;
		\item[{\rm(2)}] $\mathbb P(\rho_\gamma > 0) = 1$;
		\item[{\rm(3)}] for all $0 < h < h_{\gamma}/2$,
		\[
		\mathbb P(\rho_\gamma < h) \le C' h^{b - a\gamma_0},
		\]
		where $\gamma_0 =(\gamma+b/a)/2$, $h_\gamma= [2^{-\gamma_0}(1 - 2^{-\gamma_0})/3]^{1/(\gamma_0 - \gamma)}$ and 
		\[
		C' := \frac{CT(T\vee 1/T)2^{2b - 2a\gamma_0+1}}{1 - 2^{-(b - a\gamma_0)}}.
		\]
	\end{itemize}
\end{lemma}
\begin{proof}
    By extending \(K\) if necessary, we may assume that \(K\) is defined on
    \([0,T+1]\) and is constant on \([T,T+1]\).
	Let $\gamma_0 =(\gamma+b/a)/2\in(\gamma,b/a)$. For $n\geq 1$, define
	\[
	E_{n,\gamma_0} := \left\{ \max_{0 \le k < 2^n T} \left| K_{(k+1)/2^n} - K_{k/2^n} \right| \ge 2^{-n\gamma_0} \right\}.
	\]
	Then by Markov's inequality, we have
	\begin{equation}\label{P(En)}
		\begin{aligned}
			\mathbb P(E_{n,\gamma_0}) &\le \sum_{k=0}^{\lceil2^n T\rceil - 1}
			\frac{\mathbb{E} \left| K_{(k+1)/2^n} - K_{k/2^n} \right|^a}{2^{-an\gamma_0}} \le \lceil 2^n T\rceil \cdot C 2^{-n(1+b)+an\gamma_0}\\
			&\le 2^{n+1} T(T\vee 1/T) \cdot C 2^{-n(1+b)+an\gamma_0}= 2CT(T\vee 1/T) 2^{-n(b - a\gamma_0)}.
		\end{aligned}
	\end{equation}

	Let $0 < h < h_\gamma=[2^{-\gamma_0}(1 - 2^{-\gamma_0})/3]^{1/(\gamma_0 - \gamma)}$ and set $n_0(h) = \lfloor \log_2(1/h) \rfloor$. Then $2^{-n_0-1} < h\leq 2^{-n_0}$.
	We claim that
	\begin{equation}\label{inclusion 2.4}
	\{\bar\rho_\gamma < h\} \subseteq \bigcup_{n = n_0}^{\infty} E_{n,\gamma_0}. 
	\end{equation}
	Suppose $E_{n,\gamma_0}^c$ holds for all $n\geq n_0$. For any $s,t\in[0,T]$ satisfying $0<|t-s| \leq h$, there exists $m \ge n_0$ such that $2^{-(m+1)} < |t-s| \le 2^{-m}$. 
	Let
	$
	s_n := \lceil 2^n s \rceil 2^{-n}$ and $t_n := \lceil 2^n t \rceil 2^{-n}.
	$
	Then $s_n \downarrow s$, $t_n \downarrow t$ as $n\to\infty$, and
	\begin{equation}\label{sn-sn+1}
	|s_n - s_{n+1}| \le 2^{-(n+1)}, \quad |t_n - t_{n+1}| \le 2^{-(n+1)}. 
	\end{equation}
	By telescoping and right-continuity of $K_\cdot$, we have
	\[
	K_s = K_{s_m} + \sum_{n=m}^{\infty} (K_{s_{n+1}} - K_{s_n})\quad\text{and}\quad
	K_t = K_{t_m} + \sum_{n=m}^{\infty} (K_{t_{n+1}} - K_{t_n}).
	\]
	Hence
	\begin{equation*}
	|K_t - K_s|
	\le |K_{t_m} - K_{s_m}|
	+ \sum_{n=m}^{\infty} |K_{s_{n+1}} - K_{s_n}|
	+ \sum_{n=m}^{\infty} |K_{t_{n+1}} - K_{t_n}|. 
	\end{equation*}
	Since
	$
	|t_m - s_m| \le |t-s|+|t-t_m|+|s-s_m|\leq 3 \cdot 2^{-m},
	$
	we have
	$
	|K_{t_m} - K_{s_m}| \le 3 \cdot 2^{-m\gamma_0}. 
	$
	By \eqref{sn-sn+1}, we have
	$
	|K_{s_{n+1}} - K_{s_n}| \le 2^{-(n+1)\gamma_0}
	$
	and
	$
	|K_{t_{n+1}} - K_{t_n}| \le 2^{-(n+1)\gamma_0}. 
	$
	Then
	\[
	|K_t - K_s|
	\le 3 \cdot 2^{-m\gamma_0}
	+ 2 \sum_{n=m}^{\infty} 2^{-(n+1)\gamma_0}\leq 3\sum_{n=m}^{\infty} 2^{-n\gamma_0}= \frac{3\cdot 2^{-m\gamma_0}}{1 - 2^{-\gamma_0}}
	\le \frac{|t-s|^{\gamma_0}}{2^{-\gamma_0}(1 - 2^{-\gamma_0})/3}.
	\]
	Since $|t-s| \leq h < h_\gamma $,
	$
	|t-s|^{\gamma_0 - \gamma} < 2^{-\gamma_0}(1 - 2^{-\gamma_0})/3.
	$
	Hence $|K_t - K_s| \leq |t-s|^\gamma$, implying $\bar\rho_\gamma \ge h$. It follows that \eqref{inclusion 2.4} holds.
	
	By \eqref{P(En)} and \eqref{inclusion 2.4},
	\[
	\mathbb P(\bar\rho_\gamma < h) \le \sum_{n=n_0}^{\infty} \mathbb P(E_{n,\gamma_0})
	\le 2CT(T\vee 1/T)\sum_{n=n_0}^{\infty} 2^{-n(b - a\gamma_0)}\leq CT(T\vee 1/T) \cdot \frac{2^{1-n_0 (b - a\gamma_0)}}{1 - 2^{-(b - a\gamma_0)}}.
	\]
	Since $2^{-n_0-1} <h$,
	\begin{equation}\label{rho prob estimation}
	\mathbb P(\bar\rho_\gamma < h) \le \frac{CT(T\vee 1/T)2^{b - a\gamma_0+1}}{1 - 2^{-(b - a\gamma_0)}} \cdot h^{b - a\gamma_0}.
	\end{equation}
	Since $\rho_\gamma=\bar\rho_\gamma/2$, the same estimate with $2h$ in place of $h$ proves (3), after reducing $h_\gamma$ by a factor of $2$.
	Let $h \to 0^+$ in \eqref{rho prob estimation}. Since $b-a\gamma_0>0$, the right-hand side tends to 0 and
	\[
	\mathbb P(\bar\rho_\gamma = 0) = \lim_{h \to 0^+}\mathbb P(\bar\rho_\gamma < h) = 0,
	\]
	which leads to (2). If $|t-s|\leq\rho_\gamma=\bar\rho_\gamma/2$, then one may choose $r$ with $|t-s|<r<\bar\rho_\gamma$, and (1) follows from the definition of $\bar\rho_\gamma$. \qed
\end{proof}

The following estimate provides the initial spatial exponent.

\begin{theorem}\label{theorem Lx estimation}
	If $0<\gamma<\beta/(1+\beta),\ \varepsilon_0>0$, there is a $\rho_{\varepsilon_0,\gamma}(\omega)=\rho(\omega)>0$ $\mathbb P_{\delta_0}$-a.s. such that
	$$
	|L^{x_1}-L^{x_2}|\leq |x_1-x_2|^{\gamma}\ \text{whenever}\ |x_1-x_2|\leq \rho,\ \text{and}\ 0\leq|x_i|\leq \varepsilon_0,\ i=1,2.
	$$
	Moreover, for $n\geq 1$, there are positive $\kappa(\gamma)$, $t_\gamma$ and $C_{2.5}(\varepsilon_0,\gamma,n)$ such that for all $0<t<t_\gamma$, 
	$$\mathbb P_{\delta_0}(\{\rho<t\}\cap\{\zeta\leq n\})\leq C_{2.5}t^{\kappa},$$
	where $\zeta$ denotes the extinction time of $X$.
\end{theorem}

\begin{proof}
	Choose $q\in(1,1+\beta)$ such that $\gamma<(q-1)/q$. For $0\leq t\leq T$, by \eqref{Tanaka formula for local time}, we have
	\begin{eqnarray*}
		\mathbb E_{\delta_0}|L_t^{x_1}-L_t^{x_2}|^q\ar\leq\ar 4^{q-1}\bigg[\lambda^q\mathbb E_{\delta_0}\bigg(\int_0^T|X_s(G_\lambda^{x_1})-X_s(G_\lambda^{x_2})|\d s\bigg)^q+|G_\lambda(x_1)-G_\lambda(x_2)|^q\\
		\ar\ar+\mathbb E_{\delta_0}|X_t(G_\lambda^{x_1})-X_t(G_\lambda^{x_2})|^q+\mathbb E_{\delta_0}|M_t(G_\lambda^{x_1})-M_t(G_\lambda^{x_2})|^q\bigg]\\
		\ar=:\ar 4^{q-1}\sum_{i=1}^4 I_i.
	\end{eqnarray*}
	Let $I_i$ denote the four terms in brackets. Since $G_\lambda$ is $1$-Lipschitz, $I_2\leq |x_1-x_2|^q$ and $I_3\leq |x_1-x_2|^q\sup_{0\leq t\leq T}\mathbb E_{\delta_0}X_t(1)^q$. By Minkowski's integral inequality, we have
	\begin{eqnarray*}
		I_1\ar\leq\ar \lambda^q|x_1-x_2|^q\mathbb E_{\delta_0}\bigg(\int_0^T X_s(1)\d s\bigg)^q\leq \lambda^q|x_1-x_2|^q \bigg(\int_0^T\Big(\mathbb E_{\delta_0} X_s(1)^q\Big)^{\frac{1}{q}}\d s\bigg)^q\\
		\ar\leq\ar \lambda^q T^q|x_1-x_2|^q \sup_{0\leq t\leq T}\mathbb E_{\delta_0} X_t(1)^q.
	\end{eqnarray*}
	Applying Lemma 3.1 of \cite{LeM05} with
	$p=(3+\beta)/2$ to $\phi=G_\lambda^{x_1}-G_\lambda^{x_2}$, and using
	$|\phi(u)|\leq |x_1-x_2|$, gives
	\begin{eqnarray*}
		I_4\ar\leq\ar \mathbb E_{\delta_0}\Big[\sup_{0\leq t\leq T}|M_t(\phi)|^q\Big]\\
	\ar\leq\ar C\int_0^T \d s\int_{\mathbb R}|x_1-x_2|^q p_s(u)\d u
	+C\bigg[\int_0^T\d s\int_{\mathbb R}|x_1-x_2|^{\frac{\beta+3}{2}}p_s(u)\d u\bigg]^{\frac{2q}{3+\beta}}\\
	\ar\leq\ar C(T+T^{2q/(3+\beta)})|x_1-x_2|^q.
	\end{eqnarray*}
	Since $X_t(1)$ is a nonnegative $(1+\beta)$-stable CSBP, Theorem 6.2 of \cite{Bi76} gives, for fixed $T>0$,
	\[
	\mathbb P_{\delta_0}(X_T(1)>u)\leq C_Tu^{-(1+\beta)}
	\]
	for all sufficiently large $u$. Hence, because $q<1+\beta$,
	$\mathbb E_{\delta_0}X_T(1)^q<\infty$. Doob's inequality yields
	\begin{equation*}
		\sup_{0\leq t\leq T}\mathbb E_{\delta_0}X_t(1)^q
		\leq \mathbb E_{\delta_0}\bigg[\sup_{0\leq t\leq T}X_t(1)^q\bigg]
		\leq \bigg(\frac{q}{q-1}\bigg)^q\mathbb E_{\delta_0}X_T(1)^q<\infty.
	\end{equation*}
	By the preceding estimates, there exists $\tilde C=\tilde C(\lambda,q,T,C)$ such that
	$$
	\mathbb E_{\delta_0}|L_t^{x_1}-L_t^{x_2}|^q\leq \tilde C|x_1-x_2|^q.
	$$
	By elementary calculus, we have
	\begin{equation}\label{extinction time}
		\mathbb P_{\delta_0}\{\zeta>t\}=1-\mathbb P_{\delta_0}\{X_t(1)=0\}=1-\exp\{-(\beta t)^{-1/\beta}\}\leq (\beta t)^{-1/\beta}.
	\end{equation}
	For $n\geq 1$, define $\tilde{\Omega}_n=\{\zeta\leq n\}$ and $\tilde \Omega=\cup_{n=1}^\infty \tilde \Omega_n$. By \eqref{extinction time}, we have $\mathbb P_{\delta_{0}}(\tilde{\Omega})=1$. After translating $[-\varepsilon_0,\varepsilon_0]$ to $[0,2\varepsilon_0]$, we apply Lemma \ref{Kolmogorov's continuity criterion} to the process $L_n^{\cdot}\mathbf 1_{\tilde \Omega_n}$. Thus there exists 
	\begin{align*}
	\bar\rho_{\varepsilon_0,\gamma,n}
	&:=\sup\Big\{r\in[0,2\varepsilon_0]:
	|L_n^{x_1}-L_n^{x_2}|\mathbf1_{\tilde\Omega_n}
	\leq |x_1-x_2|^\gamma\\
	&\hspace{45mm}\text{for all }x_1,x_2\in[-\varepsilon_0,\varepsilon_0]
	\text{ with }|x_1-x_2|\leq r\Big\},\\
	\rho_{\varepsilon_0,\gamma,n}&:=\frac12\bar\rho_{\varepsilon_0,\gamma,n}.
	\end{align*}
	Then
	\[
	|L^{x_1}-L^{x_2}|\mathbf1_{\tilde\Omega_n}
	=|L_n^{x_1}-L_n^{x_2}|\mathbf1_{\tilde\Omega_n}
	\leq |x_1-x_2|^\gamma
	\]
	whenever $|x_1-x_2|\leq\rho_{\varepsilon_0,\gamma,n}$ and
	$x_1,x_2\in[-\varepsilon_0,\varepsilon_0]$. The variable $\rho_{\varepsilon_0,\gamma,n}$ depends only on $L_n^{\cdot}\mathbf 1_{\tilde \Omega_n}$. Since $L_n^{\cdot}\mathbf 1_{\tilde \Omega_n}=L_{n+1}^{\cdot}\mathbf 1_{\tilde \Omega_{n+1}}$ on $\tilde \Omega_n$, we have $\rho_{\varepsilon_0,\gamma,n}=\rho_{\varepsilon_0,\gamma,n+1}$ there. On $\tilde \Omega_n^c$, $L_n^{\cdot}\mathbf 1_{\tilde \Omega_n}\equiv0$, hence $\bar\rho_{\varepsilon_0,\gamma,n}=2\varepsilon_0$ and $\rho_{\varepsilon_0,\gamma,n}=\varepsilon_0$. Thus $\rho_{\varepsilon_0,\gamma}:=\lim_{n\to\infty}\rho_{\varepsilon_0,\gamma,n}\mathbf 1_{\tilde \Omega_n}$ exists and
	$$
	|L^{x_1}-L^{x_2}|\leq |x_1-x_2|^\gamma\ \text{whenever}\ |x_1-x_2|\leq \rho_{\varepsilon_0,\gamma},\ 0\leq|x_i|\leq \varepsilon_0,\ i=1,2
	$$
	holds on $\tilde{\Omega}$ and $\mathbb P_{\delta_{0}}(\rho_{\varepsilon_0,\gamma}>0)=1$. For $n\geq 1$, by (3) in Lemma \ref{Kolmogorov's continuity criterion}, there exist positive constants $\kappa(\gamma)$, $t_\gamma<\varepsilon_0$ and $C_{2.5}(\varepsilon_0,\gamma,n)$ such that for all $0<t<t_\gamma$, 
	\[
	\mathbb P_{\delta_0}(\{\rho_{\varepsilon_0,\gamma}<t\}\cap\{\zeta\leq n\})=\mathbb P_{\delta_0}(\rho_{\varepsilon_0,\gamma,n}<t)\leq C_{2.5}t^{\kappa}. \tag*{$\square$}
	\]
\end{proof}

\section{Proof of Theorem \ref{random interval}}

\setcounter{equation}{0}

In this section, we consider the super-Brownian motion with $(1+\beta)$-stable branching mechanism starting from $y_0\delta_{0}$. For $r>0$, let $Y_r\delta_r$ denote the exit measure from $(-\infty,r)$, and set $Y_0=y_0$.  By the special Markov property (see Corollary 3.9 in \cite{RiR24}) and spatial translation invariance, for every bounded measurable $\psi:\mathbb R_+\to\mathbb R$ and $0<r_1<r_2$,
\begin{equation}\label{special Markov}
	\mathbb E_{y_0\delta_0}\!\left(\psi(Y_{r_2})\,\middle|\,\mathcal F^Y_{r_1}\right)
	=\mathbb E_{Y_{r_1}\delta_0}\!\left(\psi(Y_{r_2-r_1})\right),
	\qquad \mathbb P_{y_0\delta_0}\text{-a.s.},
\end{equation}
where $\mathcal F^Y_{r_1}=\sigma(Y_s:0\leq s\leq r_1)$. Recall the definition of an $\alpha$-stable continuous-state branching process.

\begin{definition}
	Let $1<\alpha<2$. A $[0,\infty)$-valued Markov process is called a $\alpha$-stable continuous-state branching process ($\alpha$-stable CSBP) with scaling constant $c_0>0$ if it has transition semigroup $(Q_t)_{t\geq 0}$ defined by
	\begin{equation}
		\int_0^\infty e^{-\lambda y}Q_t(x,\d y)=e^{-xu^\lambda(t)}, \quad \lambda\geq 0,x\geq 0,
	\end{equation}
	where $t\mapsto u^\lambda(t)$ is the unique positive solution of 
	\begin{equation}\label{stable backward equation}
		\frac{\d u^\lambda(t)}{\d t}=-c_0[u^\lambda(t)]^\alpha,\quad u^\lambda(0)=\lambda.
	\end{equation}
\end{definition}

\begin{proposition}\label{prop Y SCSBP}
	Under $\mathbb P_{y_0\delta_0}$, there is a c\`adl\`ag strong Markov version of $Y$ (also denoted $Y$) which is a $(1+\beta/2)$-stable CSBP starting at $y_0$ with parameter $c_0=2/\sqrt{\beta+2}$. The process is a spectrally positive martingale, $0$ is absorbing, and if
	$$
	R=\inf\{r>0:Y_r\wedge Y_{r-}=0\},
	$$
	then $Y_r=0$ for all $r\geq R=\inf \{r>0:Y_r=0\}<\infty\ \mathbb P_{y_0\delta_0}$-a.s.
\end{proposition}
\begin{proof}
	Let $\mathcal{F}_r^Y := \sigma(Y_s : 0 \le s \le r)$. For $\lambda > 0$ and $r > 0$, define
	\begin{equation}\label{U lambda r x}
		U^{\lambda,r}(x)
		= U^{\lambda,0}(x-r)
		= \left[
		\lambda^{-\frac{\beta}{2}}
		+ \frac{\beta}{\sqrt{\beta+2}}(r-x)
		\right]^{-\frac{2}{\beta}},
		\qquad x \le r.
	\end{equation}
    By \cite[Theorem~4.3.3]{DuL02},  the half-line exit-measure Laplace exponent
    \[
    u(x):=\mathbb N_x\!\left(1-e^{-\lambda Y_r}\right),\qquad x<r,
    \]
    is a bounded nonnegative solution of
    \[
    \frac12 u''=u^{1+\beta}\quad\text{on }(-\infty,r),
    \qquad u(r)=\lambda.
    \]
    As in the uniqueness argument of Lemma~\ref{explicit V lambda}, boundedness forces
    $u(x)\to0$ as $x\to-\infty$, and integration of the ODE gives
    \[
    u(x)=
    \left[
    \lambda^{-\frac{\beta}{2}}
    +\frac{\beta}{\sqrt{\beta+2}}(r-x)
    \right]^{-\frac{2}{\beta}}
    =:U^{\lambda,r}(x).
    \]
    Hence, by the Poisson representation under $\mathbb P_{y_0\delta_0}$,
	\begin{equation}\label{Yr}
		\mathbb{E}_{y_0\delta_0}\left(e^{-\lambda Y_r}\right)
		= \exp\left\{-y_0 U^{\lambda,r}(0)\right\}.
	\end{equation}
	The same argument shows more generally that
	\[
	\mathbb{E}_{z\delta_0}\left(e^{-\lambda Y_r}\right)
	= \exp\left\{-z U^{\lambda,r}(0)\right\},
	\qquad z \ge 0.
	\]
	Set
	\[
	u^\lambda(t) := U^{\lambda,t}(0)
	= \left[
	\lambda^{-\frac{\beta}{2}}
	+ \frac{\beta}{\sqrt{\beta+2}}t
	\right]^{-\frac{2}{\beta}}.
	\]
	Since $u^\lambda(t) \to \lambda$ as $t \downarrow 0$, the preceding Laplace transform implies that $Y_t \to z$ in law, and hence in probability, under $\mathbb{P}_{z\delta_0}$.
	
	We verify right-continuity in probability. Fix $r > 0$. For $\theta, \lambda > 0$ and $h > 0$, the special Markov property \eqref{special Markov} gives
	\begin{align*}
		\mathbb{E}_{y_0\delta_0}\left(e^{-\theta Y_r - \lambda Y_{r+h}}\right)
		&= \mathbb{E}_{y_0\delta_0}\left[
		e^{-\theta Y_r} \mathbb{E}_{Y_r\delta_0}\left(e^{-\lambda Y_h}\right)
		\right] \\
		&= \mathbb{E}_{y_0\delta_0}\left(
		\exp\left\{ -\bigl(\theta + u^\lambda(h)\bigr)Y_r \right\}
		\right).
	\end{align*}
	As $h \downarrow 0$, the last expression converges to $\mathbb{E}_{y_0\delta_0}\left(e^{-(\theta+\lambda)Y_r}\right)$, which is the joint Laplace transform of $(Y_r, Y_r)$. It follows that
	\[
	(Y_r, Y_{r+h}) \longrightarrow (Y_r, Y_r) \quad\text{in law}.
	\]
	Since $\tilde f(x,y):=y-x$ is continuous, we have $Y_{r+h} - Y_r \to 0$ in law and thus in probability. Together with the convergence at $0$, this shows that $Y$ is right-continuous in probability on $[0, \infty)$.
	
	Since $Y_0 = y_0$ is deterministic, \eqref{special Markov} also holds for $r_1 = 0$. Indeed, for every bounded measurable function $\psi$ and $t > 0$,
	\[
	\mathbb{E}_{y_0\delta_0}\left( \psi(Y_t) \,\middle|\, \mathcal{F}_0^Y \right)
	= \mathbb{E}_{y_0\delta_0}\left(\psi(Y_t)\right)
	= \mathbb{E}_{Y_0\delta_0}\left(\psi(Y_t)\right).
	\]
	Thus $Y = (Y_r)_{r \ge 0}$ is a Markov process starting from $y_0$.
	
	Let $Q_t(z, \mathrm{d}y)$ denote its transition kernel. The preceding Laplace transform gives
	\[
	\int_{[0,\infty)} e^{-\lambda y} Q_t(z, \mathrm{d}y)
	= \exp\left\{-z u^\lambda(t)\right\}.
	\]
	Moreover,
	\begin{equation}\label{u lambda equation}
		\frac{\mathrm{d}}{\mathrm{d}t}u^\lambda(t)
		= -\frac{2}{\sqrt{\beta+2}} \bigl[u^\lambda(t)\bigr]^{1+\frac{\beta}{2}},
		\qquad u^\lambda(0) = \lambda.
	\end{equation}
	Hence $Y$ has the transition semigroup of a $(1+\beta/2)$-stable CSBP with
	scaling constant $2/\sqrt{\beta+2}$.  Since this semigroup is Feller and the
	state space $[0,\infty)$ is locally compact and separable, the standard Feller
	regularization theorem, see \cite[Chapter~4, Theorem~2.7]{EtK86}, provides a
	c\`adl\`ag strong Markov version.  We use this version from now on; at every
	deterministic radius it agrees almost surely with the geometrically defined
	exit mass.  In particular, all subsequent pathwise statements and stopping
	times involving $Y$, including the stopping times $\tau_n$ in Section~5, refer
	to this version.  By the standard spectral-positivity property of stable
	CSBPs (see, e.g., \cite{Bi76}), $Y$ has no negative jumps and $0$ is absorbing.
	
	The process $Y$ is also a martingale. For every $t \ge 0$,
	\[
	\frac{u^\lambda(t)}{\lambda}
	= \left[
	1 + \frac{\beta}{\sqrt{\beta+2}} t \lambda^{\frac{\beta}{2}}
	\right]^{-\frac{2}{\beta}}
	\longrightarrow 1 \quad\text{as }\lambda \downarrow 0.
	\]
	For $0 \le r_1 < r_2$, the special Markov property implies
	\[
	\mathbb{E}_{y_0\delta_0}\left( e^{-\lambda Y_{r_2}} \,\middle|\, \mathcal{F}_{r_1}^Y \right)
	= \exp\left\{ -Y_{r_1} u^\lambda(r_2 - r_1) \right\}.
	\]
	Consequently,
	\begin{equation}\label{conditional Laplace martingale}
		\mathbb{E}_{y_0\delta_0}\left( \frac{1 - e^{-\lambda Y_{r_2}}}{\lambda} \,\middle|\, \mathcal{F}_{r_1}^Y \right)
		= \frac{1 - \exp\left\{-Y_{r_1} u^\lambda(r_2 - r_1)\right\}}{\lambda}.
	\end{equation}
	Letting $\lambda \downarrow 0$ along a sequence, conditional monotone convergence shows that the left-hand side of \eqref{conditional Laplace martingale} converges to $\mathbb{E}_{y_0\delta_0}(Y_{r_2} \mid \mathcal{F}_{r_1}^Y)$, while the right-hand side converges to $Y_{r_1}$. Hence
	\[
	\mathbb{E}_{y_0\delta_0}\left( Y_{r_2} \,\middle|\, \mathcal{F}_{r_1}^Y \right)
	= Y_{r_1}, \qquad \mathbb{P}_{y_0\delta_0}\text{-a.s.}
	\]
	Thus the c\`adl\`ag version fixed above is a nonnegative martingale.
	
	Let
	\[
	R := \inf\{r > 0 : Y_r \wedge Y_{r-} = 0\}.
	\]
	Since $0$ is absorbing and $Y$ has no negative jumps, $Y_r=0$ for every $r\geq R$, and
	\[
	R=\inf\{r>0:Y_r=0\}.
	\]
	For $r > 0$, letting $\lambda \to \infty$ in \eqref{Yr} gives
	\begin{align*}
		\mathbb{P}_{y_0\delta_0}(R \le r)
		&= \mathbb{P}_{y_0\delta_0}(Y_r = 0) \\
		&= \lim_{\lambda \to \infty} \mathbb{E}_{y_0\delta_0}\left(e^{-\lambda Y_r}\right) \\
		&= \exp\left\{ -y_0 \left( \frac{\beta r}{\sqrt{\beta+2}} \right)^{-\frac{2}{\beta}} \right\}.
	\end{align*}
	Letting $r \to \infty$, we conclude that $R < \infty$, $\mathbb{P}_{y_0\delta_0}$-a.s. \qed
\end{proof}

\noindent\textbf{Proof of Theorem \ref{random interval}.}
Let $R_n=n\wedge \inf\{r\geq 0:Y_r\leq 2^{-n}\}$. Then $R_n\uparrow R$ as $n\to\infty$. Choose $\eta>\beta/2$ and take $H=(1+2/\beta)\eta$. For $i\in\mathbb Z_+$, let $x_{i,n}=i2^{-n\eta}$ and $I_{i,n}=(x_{i,n},x_{i+1,n}]$.

Fix $x\in I_{i,n}$. On $\{x<R_n\}$ one has $x_{i,n}<R_n$ and hence $Y_{x_{i,n}}>2^{-n}$. Moreover, because the spatial motion is continuous and $x>x_{i,n}$, occupation at the point $x$ can only be generated by genealogical paths after they have exited $(-\infty,x_{i,n})$. By the special Markov property of exit measures (see Corollary 3.9 of \cite{RiR24}), conditionally on the exit measure $Y_{x_{i,n}}\delta_{x_{i,n}}$, these descendants are distributed as a superprocess started from $Y_{x_{i,n}}\delta_{x_{i,n}}$. Therefore
\begin{eqnarray*}
	\mathbb P_{\delta_0}(x<R_n,L^x\leq 2^{1-nH})\ar\leq\ar \mathbb E_{\delta_0}\Big(\mathbf 1_{\{Y_{x_{i,n}}\geq 2^{-n}\}}\mathbb P_{Y_{x_{i,n}}\delta_{x_{i,n}}}(L^x\leq 2^{1-nH})\Big)\\
	\ar\leq\ar \mathbb E_{\delta_0}\Big(\mathbf 1_{\{Y_{x_{i,n}}\geq 2^{-n}\}}\mathbb E_{Y_{x_{i,n}}\delta_{x_{i,n}}}(\exp\{2-2^{nH}L^x\})\Big)\\
	\ar=\ar e^2\mathbb E_{\delta_0}\Big(\mathbf 1_{\{Y_{x_{i,n}}\geq 2^{-n}\}}\exp\{-Y_{x_{i,n}}V^{2^{nH}}(x-x_{i,n})\}\Big)\\
	\ar\leq\ar e^2\mathbb E_{\delta_0}\Big(\mathbf 1_{\{Y_{x_{i,n}}\geq 2^{-n}\}}\exp\{-2^{-n}V^{2^{nH}}(x-x_{i,n})\}\Big)\\
	\ar\leq\ar e^2\exp\{-2^{2\eta n/\beta-n}V^{2^{nH-n(1+2/\beta)\eta}}((x-x_{i,n})/2^{-n\eta}) \}\\
	\ar\leq\ar e^2\exp\{-2^{2\eta n/\beta-n}V^1(1)\}.
\end{eqnarray*}
Here we use Lemma \ref{lemma V lambda} for the third line. By \eqref{precise V lambda},
$$
V^\lambda(x)=r^{-2/\beta}V^{\lambda r^{1+2/\beta}}(x/r),\quad \forall x\in \mathbb R, r>0,
$$
implying the above fifth line. And for the last line, we use the fact that $V^1$ is decreasing by \eqref{precise V lambda} and the definition of $H$. Take $k\in\mathbb N$ satisfying $k\geq 2H(1+\beta)/\beta$. Fix $\varepsilon_0>0$ and let $\rho=\rho_{\varepsilon_0,\beta/(2+2\beta)}$ be as in Theorem \ref{theorem Lx estimation} with parameter $\gamma=\beta/(2+2\beta)$. Define
$$
\Lambda_n=\{L^x\leq 2^{1-nH}\ \text{for some}\ x\in[0,R_n)\cap\{j 2^{-n(\eta+k)}:j\in\mathbb Z_+\}\}
$$
and
$$
\tilde \Lambda_n=\{L^x\leq 2^{-nH}\ \text{for some}\ x\in[0,R_n\wedge\varepsilon_0)\}.
$$
Since $R_n\leq n$, we have
\begin{equation}\label{pn}
	\mathbb P_{\delta_0}(\Lambda_n)\leq e^2(n2^{n(\eta+k)}+1)\exp\{-2^{2\eta n/\beta-n}V^1(1)\}=:p_n.
\end{equation}
Assume $\omega$ is chosen in $\Lambda_n^c\cap\{\rho\geq 2^{-n(\eta+k)}\}$. For $x\in[0,R_n\wedge\varepsilon_0)$, there exists  $j\in\mathbb Z_+$ such that $x\in[j2^{-n(\eta+k)},(j+1)2^{-n(\eta+k)})$. Then by Theorem \ref{theorem Lx estimation}, we have
$$
L^x\geq L^{j2^{-n(\eta+k)}}-2^{-n(\eta+k)\beta/(2+2\beta)}>2^{1-nH}-2^{-nH}=2^{-nH},
$$
implying $\Lambda_n^c\cap\{\rho\geq 2^{-n(\eta+k)}\}\subset \tilde \Lambda_n^c$.
For $l\geq 1$, recall that $\tilde{\Omega}_l=\{\zeta\leq l\}$, $\tilde \Omega=\cup_{l=1}^\infty \tilde \Omega_l$ and $\mathbb P_{\delta_{0}}(\tilde{\Omega})=1$.
Then by \eqref{pn} and Theorem \ref{theorem Lx estimation}, for any $l\geq 1$ and sufficiently large $n$, we have
\begin{eqnarray*}
	\mathbb P_{\delta_0}(\tilde \Lambda_n\cap\tilde{\Omega}_l)\leq \mathbb P_{\delta_0}( \Lambda_n\cap\tilde{\Omega}_l)+\mathbb P_{\delta_0}(\{\rho<2^{-n(\eta+k)}\}\cap\tilde{\Omega}_l)
	\leq p_n+C_l2^{-n\kappa(\eta+k)}.
\end{eqnarray*}
The right-hand side is summable in $n$. Hence, for every fixed $l\geq1$, the Borel--Cantelli lemma gives
\[
\mathbb P_{\delta_0}\left(\tilde\Omega_l\cap\limsup_{n\to\infty}\tilde\Lambda_n\right)=0.
\]
Since $\tilde\Omega=\bigcup_{l\geq1}\tilde\Omega_l$, it follows that
\[
\mathbb P_{\delta_0}\left(\tilde\Omega\cap\limsup_{n\to\infty}\tilde\Lambda_n\right)=0.
\]
As $\mathbb P_{\delta_0}(\tilde\Omega)=1$, there exists an almost surely finite random integer $\tilde N(\varepsilon_0)$ such that, for all $n\geq\tilde N(\varepsilon_0)$,
\[
\inf_{x\in[0,R_n\wedge\varepsilon_0)}L^x\geq 2^{-nH}.
\]
Take $\varepsilon_0=m$, $m\in\mathbb N$, and intersect the resulting probability-one events. For any $x\in[0,R)$, choose $m>x$ and then $n$ sufficiently large so that $x<R_n$. This yields $L^x>0$ for every $x\in[0,R)$ almost surely.

It remains to show $L^x=0$ for all $x\geq R$.
For $0<r<q$, by special Markov property and Lemma \ref{lemma V lambda}, we have
$$
\mathbb P_{\delta_0}(Y_r=0,L^q>0)=\mathbb E_{\delta_0}(\mathbf 1_{\{Y_r=0\}}\mathbb P_{Y_r\delta_r}(L^q>0))=0.
$$
Thus we can find a $\mathbb P_{\delta_0}$-null set $A_{null}$ such that, on $A_{null}^c$, for all rational $0<r<q$, $Y_r=0$ implies $L^q=0$, $x\mapsto L^x$ is continuous on $(0,\infty)$ and $Y_a=0$ holds for all $a\geq R$. Fix $\omega\in A_{null}^c$. For any rational $q>R$, there exists a rational $r\in(R,q)$. Then $Y_r=0$, which implies $L^q=0$. For $q>R$ irrational, choose rational $q_n\downarrow q$. Since $L^{q_n}=0$, continuity gives $L^q=0$. 

Thus 
$$\{x\geq 0:L^x>0\}=[0,R)$$
with probability 1, and hence $\mathsf R=R$. By symmetry, with probability 1, there also exists a random variable $-\infty<\mathsf L<0$ such that 
$$\{x\leq 0:L^x>0\}=(\mathsf L,0].$$ 
Therefore $\{x:L^x>0\}=(\mathsf L,\mathsf R)$ holds a.s. \qed

\section{Upper bound of the local time near the boundary}

\setcounter{equation}{0}

We prove the upper boundary estimate by a Tanaka formula and a dyadic bootstrap. Set $g_x(y)=|y-x|$, so $\frac{\d^2}{\d y^2}g_x=2\delta_x$ distributionally. The following approximation is based on \cite[Proposition 2.4]{Ho18}.

Let \(J\in C_c^\infty(\mathbb R)\) denote the standard mollifier
\[
J(x):=C_1\exp\left(\frac{1}{x^2-1}\right)
\mathbf 1_{\{|x|<1\}},
\]
where \(C_1>0\) is chosen so that
\(\int_{\mathbb R}J(x)\,\d x=1\). For \(N\geq1\), define
\begin{equation}\label{XNx}
	\chi_N(x)
	:=\int_{\mathbb R}\mathbf 1_{\{|x-y|<N\}}J(y)\,\d y
	=\int_{x-N}^{x+N}J(y)\,\d y.
\end{equation}
Then \(\chi_N\in C_c^\infty(\mathbb R)\),
\(
\chi_N=1\text{ on }[-N+1,N-1],
\text{ and }
\operatorname{supp}(\chi_N)\subseteq[-N-1,N+1].
\)
Since \(P_\varepsilon g_x\in C^\infty(\mathbb R)\), it follows that
\(
P_\varepsilon g_x\cdot\chi_N
\in C_c^\infty(\mathbb R)\subseteq C_b^2(\mathbb R).
\)
Applying the martingale problem and taking a countable intersection
over \(N\), we deduce that, $\mathbb P_{\delta_0}$-a.s., for every \(N\geq1\),
\begin{equation}\label{cutoff martingale}
	X_t(P_\varepsilon g_x\cdot\chi_N)
	=P_\varepsilon g_x(0)\chi_N(0)
	+M_t(P_\varepsilon g_x\cdot\chi_N)
	+\int_0^t
	X_s\left(
	\frac{\Delta}{2}(P_\varepsilon g_x\cdot\chi_N)
	\right)\d s,
	\qquad t\geq0.
\end{equation}

Fix \(T>0\). By Theorem EP2 in \cite{EnP06}, \(X\) has the compact
support property. Thus, with probability one, there exists a random variable \(R_T<\infty\)
such that
\[
\bigcup_{0\leq s\leq T}\operatorname{supp}(X_s)
\subseteq[-R_T,R_T].
\]
For \(N>R_T+1\),
\[
X_t(P_\varepsilon g_x\cdot\chi_N)
=X_t(P_\varepsilon g_x),
\qquad 0\leq t\leq T,
\]
and
\[
X_s\left(
\frac{\Delta}{2}(P_\varepsilon g_x\cdot\chi_N)
\right)
=
X_s\left(
\frac{\Delta}{2}P_\varepsilon g_x
\right),
\qquad 0\leq s\leq T.
\]
Moreover, \(\chi_N(0)=1\) for every \(N\geq1\). Hence, \eqref{cutoff martingale} implies that, for all sufficiently large \(N\),
\begin{equation}\label{eventual martingale identity}
	M_t(P_\varepsilon g_x\cdot\chi_N)
	=X_t(P_\varepsilon g_x)-P_\varepsilon g_x(0)
	-\int_0^tX_s\left(\frac{\Delta}{2}P_\varepsilon g_x\right)\d s,
	\qquad 0\leq t\leq T, \quad \text{a.s.}
\end{equation}

Lemma 1.7 in \cite{FMW10} ensures that $M_t(P_\varepsilon g_x)$ is well-defined, and Lemma 3.1 in \cite{LeM05} gives
\begin{align}\label{M_t estimation 4}
	&\mathbb E_{\delta_0}\left[
	\left(
	\sup_{t\leq T}
	\left|
	M_t(P_\varepsilon g_x\cdot\chi_N)
	-M_t(P_\varepsilon g_x)
	\right|
	\right)^{1+\beta/2}
	\right] \notag\\
	&\quad\leq C_2\left[
	\int_0^T\d s\int_{\mathbb R}
	\left|
	P_\varepsilon g_x(y)\chi_N(y)
	-P_\varepsilon g_x(y)
	\right|^{(3+\beta)/2}
	p_s(y)\,\d y
	\right]^{(2+\beta)/(3+\beta)} \notag\\
	&\qquad+C_2\int_0^T\d s\int_{\mathbb R}
	\left|
	P_\varepsilon g_x(y)\chi_N(y)
	-P_\varepsilon g_x(y)
	\right|^{1+\beta/2}
	p_s(y)\,\d y.
\end{align}
Since $0\leq\chi_N\leq1$ and $\chi_N\to1$ pointwise, both terms on the right-hand side vanish as $N\to\infty$. Therefore,
\[
\sup_{t\leq T}
\left|
M_t(P_\varepsilon g_x\cdot\chi_N)
-M_t(P_\varepsilon g_x)
\right|
\longrightarrow0
\quad\text{in } L^{1+\beta/2}(\mathbb P_{\delta_0}), \text{ and so in probability}.
\]
Along a subsequence $N_k\to\infty$,
$$
M_t(P_\varepsilon g_x) =X_t(P_\varepsilon g_x)-P_\varepsilon g_x(0)-\int_0^tX_s\left(\frac{\Delta}{2}P_\varepsilon g_x\right)\d s,     \qquad 0\leq t\leq T
$$
holds $\mathbb{P}_{\delta_0}$-a.s.

Moreover, since $T$ is arbitrary, by taking a countable intersection over $m \in \mathbb{N}$ (setting $T=m$), we extend this identity to all $t \geq 0$ to conclude that, $\mathbb{P}_{\delta_0}$-a.s.,

\begin{equation}\label{Xt Pep gx}
	X_t(P_\varepsilon g_x)=P_\varepsilon g_x(0)+M_t(P_\varepsilon g_x)+\int_0^tX_s\left(\frac{\Delta}{2}P_\varepsilon g_x\right)\d s,\qquad t\geq 0.
\end{equation}

This yields the Tanaka formula.

\begin{proposition}\label{prop Tanaka 4.1}
	For any $x\neq0$, 
	\begin{equation}\label{Tanaka formula 4.1}
		L_t^x=-|x|+X_t(g_x)-M_t(g_x),\quad \forall t\geq 0
	\end{equation}
	holds $\mathbb P_{\delta_0}$-a.s.
\end{proposition}
\begin{proof}
	From \eqref{Xt Pep gx}, it suffices to pass to a subsequential limit as $\varepsilon\downarrow0$. Since
	\begin{align*}
		|P_\varepsilon g_x(y) - g_x(y)|
		&= \left| \int_{\mathbb{R}} \big(|z - x| - |y - x|\big) p_\varepsilon(z - y) dz \right| \\
		&\leq \int_{\mathbb{R}} |z - y| p_\varepsilon(z - y) dz.
	\end{align*}
	With $u=z-y$,
	\begin{align*}
		\int_{\mathbb{R}} |u| p_\varepsilon(u) du 
		=\frac{2}{\sqrt{2\pi\varepsilon}}
		\int_0^\infty
		u\exp\left(-\frac{u^2}{2\varepsilon}\right)\,\d u 
		= \sqrt{\frac{2\varepsilon}{\pi}} \leq \varepsilon^{1/2},
	\end{align*}
	which implies
	\begin{equation}\label{tanaka0}
		\sup_{y\in\mathbb R}
		\bigl|P_\varepsilon g_x(y)-g_x(y)\bigr|
		\leq \varepsilon^{1/2}.
	\end{equation}
	Now if we take \(y=0\) in \eqref{tanaka0} and recall that \(g_x(0)=|x|\), then
	\[
	\bigl|P_\varepsilon g_x(0)-|x|\bigr|
	=
	\bigl|P_\varepsilon g_x(0)-g_x(0)\bigr|
	\leq \varepsilon^{1/2},
	\]
	and hence
	\begin{equation}\label{tanaka1}
		P_\varepsilon g_x(0)\longrightarrow |x|
		\qquad\text{as }\varepsilon\downarrow0.
	\end{equation}
	
	Next for every \(T\geq 0\) and  \(t\in[0,T]\), we get the bound
	\begin{align}
		\bigl|X_t(P_\varepsilon g_x)-X_t(g_x)\bigr|
		=
		\bigl|X_t(P_\varepsilon g_x-g_x)\bigr| \notag\
		\leq
		X_t\bigl(|P_\varepsilon g_x-g_x|\bigr) \notag\
		\leq
		\varepsilon^{1/2}X_t(1).
	\end{align}
	Therefore,
	\begin{equation}\label{tanaka2}
		\sup_{t\leq T}
		\bigl|X_t(P_\varepsilon g_x)-X_t(g_x)\bigr|
		\leq
		\varepsilon^{1/2}\sup_{t\leq T}X_t(1)\longrightarrow0,
		\qquad \mathbb P_{\delta_0}\text{-a.s.},
	\end{equation}
	since
	\(
	\sup_{t\leq T}X_t(1)<\infty,
	\mathbb P_{\delta_0}\text{-a.s.}
	\)
	
	Lemma 1.7 in \cite{FMW10} implies that $M_t(g_x)$ is well-defined and Lemma 3.1 in \cite{LeM05} gives
	\begin{equation}\label{tanaka3}
		\mathbb{E}_{\delta_0}\left[\left(\sup_{t\leq T}\left|M_t(P_\varepsilon g_x)-M_t(g_x)\right|\right)^{1+\beta/2}\right] \to 0 \text{ as } \varepsilon\downarrow0.
	\end{equation}

	For the convergence of last term, one first check  that $\frac{\Delta}{2} P_\varepsilon g_x(y) = p_\varepsilon(y-x) =: p_\varepsilon^x(y)$ and the proof of Theorem~1.2 in \cite{Xi05} gives
	\begin{equation}\label{tanaka4}
		\sup_{t\leq T}\bigg\vert{}\int_0^t X_s(p_\varepsilon^x)\d s - L_t^x\bigg\vert{} \overset{p}{\to} 0.
	\end{equation}

	Finally,  by passing to a suitable subsequence
	\(\varepsilon_n\downarrow0\), all the convergences in
	\eqref{tanaka1}--\eqref{tanaka4} hold simultaneously almost surely.
	We may therefore pass to the limit and obtain
	\eqref{Tanaka formula 4.1} for all $t\in [0,T]$ almost surely. Since $T$ is arbitrary, by taking a countable intersection over $m \in \mathbb{N}$ (setting $T=m$), we extend this identity to all $t \geq 0$ a.s., completing the proof.\qed

\end{proof}

Let $\sgn(x)=x/|x|$ for $x\neq0$ and $\sgn(0)=0$. For $x<0$, use the convention
\begin{equation*}
	\int_0^x f(x)\d x:=-\int_x^0 f(x)\d x.
\end{equation*}

\begin{proposition}\label{Tanaka integral formula}
	Fix $t\geq0$. Then
	\begin{equation}
		L_t^x-L_t^0=\int_0^x D_t^y \d y,\qquad x\in \mathbb R,
	\end{equation}
	holds $\mathbb P_{\delta_0}$-a.s., where
	\begin{equation}\label{Dtx formula}
		D_t^y:=-\sgn(y)+X_t(\sgn(y-\cdot))-M_t(\sgn(y-\cdot)),\ y\in \mathbb R.
	\end{equation}
\end{proposition}

\begin{proof}
	The case $t=0$ is immediate. Fix $t>0$.
	Approximating $(y,z)\mapsto\sgn(y-z)$ by jointly measurable simple functions and using Lemma 3.1 of \cite{LeM05}, fix a jointly measurable version of $(y,\omega)\mapsto M_t(\sgn(y-\cdot))(\omega)$. Hence $(y,\omega)\mapsto D_t^y(\omega)$ is jointly measurable.
	
	For every $n\in\mathbb N$, Fubini's theorem gives
	\begin{align*}
		\mathbb E_{\delta_0}
		\left[\int_{-n}^n|D_t^y|\,\mathrm{d} y\right]
		&\leq
		2n+2n\mathbb E_{\delta_0}[X_t(1)]
		+
		\int_{-n}^n
		\mathbb E_{\delta_0}
		\left|M_t\bigl(\sgn(y-\cdot)\bigr)\right|
		\,\mathrm{d} y
		<\infty,
	\end{align*}
	where the inequality follows from
	\[
	\left|X_t\bigl(\sgn(y-\cdot)\bigr)\right|\leq X_t(1)
	\quad\text{and}\quad
	\sup_{y\in\mathbb R}
	\mathbb E_{\delta_0}
	\left|M_t\bigl(\sgn(y-\cdot)\bigr)\right|<\infty.
	\]
	The second estimate is guaranteed by Lemma 3.1 in \cite{LeM05}.
	Consequently, $D_t^\cdot \in L^1_{\mathrm{loc}}(\mathbb R)$, $\mathbb P_{\delta_0}$-a.s.
	
	Set
	\[
	H_t^x
	:=
	L_t^x-L_t^0-\int_0^xD_t^y\,\d y,
	\qquad x\in\mathbb R.
	\]
	Since $L_t^\cdot$ is continuous and $D_t^\cdot\in L^1_{\mathrm{loc}}(\mathbb R)$ a.s., $H_t^\cdot$ is continuous a.s. It suffices to show $H_t^x=0$ for all $x$. For $x\neq0$,
	\begin{equation}\label{H increment estimate}
		\lim_{h\to0}
		\frac{
			\mathbb E_{\delta_0}|H_t^{x+h}-H_t^x|
		}{|h|}
		=0,
		\qquad x\neq0.
	\end{equation}
	Fix $x\neq0$. For $|h|$ sufficiently small, $x$ and $x+h$
	have the same sign. By the definition of $H_t^x$,
	\begin{align*}
		\frac{
			\mathbb E_{\delta_0}|H_t^{x+h}-H_t^x|
		}{|h|}
		&\leq
		\mathbb E_{\delta_0}
		\left|
		\frac{L_t^{x+h}-L_t^x}{h}-D_t^x
		\right|
		+
		\mathbb E_{\delta_0}
		\left|
		\frac1h\int_x^{x+h}(D_t^y-D_t^x)\,\d y
		\right|\\
		&=:I_1(h)+I_2(h).
	\end{align*}
	
	Since
	\(
	\frac{|x+h|-|x|}{h}=\sgn(x),
	\)
	together with Proposition \ref{prop Tanaka 4.1}, we deduce that
	\[
	\frac{L_t^{x+h}-L_t^x}{h}-D_t^x
	=
	X_t(\psi_h^x)-M_t(\psi_h^x),
	\]
	where
	\[
	\psi_h^x(z)
	:=
	\frac{g_{x+h}(z)-g_x(z)}{h}
	-\sgn(x-z).
	\]
	On one hand,
	\(
	|\psi_h^x(z)|
	\leq
	2\cdot\mathbf 1_{\{
		x\wedge(x+h)\leq z\leq x\vee(x+h)
		\}},
	\)
	and thus
	\begin{equation}\label{X_psi}
		\mathbb E_{\delta_0}|X_t(\psi_h^x)|
		\leq
		2\int_{x\wedge(x+h)}^{x\vee(x+h)}
		p_t(z)\,\d z
		\to 0 \ \text{as}\ h\to 0.
	\end{equation}
	On the other hand, Lemma 3.1 of \cite{LeM05} gives
	\begin{align*}
	&\mathbb E_{\delta_0}\left[\sup_{0\leq s\leq t}|M_s(\psi_h^x)|^{1+\beta/2}\right]\\
	&\quad\leq C\left[\int_0^t\!\d s\int_{I_h}|\psi_h^x(z)|^{(3+\beta)/2}p_s(z)\,\d z\right]^{(2+\beta)/(3+\beta)}
	+C\int_0^t\!\d s\int_{I_h}|\psi_h^x(z)|^{1+\beta/2}p_s(z)\,\d z,
	\end{align*}
	where $I_h=[x\wedge(x+h),x\vee(x+h)]$. Since $|\psi_h^x|\leq2\mathbf1_{I_h}$ and $|I_h|\to0$, the right-hand side tends to zero. H\"older's inequality therefore yields
	\begin{equation}\label{M_psi}
		\mathbb E_{\delta_0}|M_t(\psi_h^x)| \to 0 \ \text{as}\ h\to 0.
	\end{equation}
	Thus $I_1(h)\to0$.

	For $I_2(h)$, set
	\[
	\eta_{x,y}(z)
	:=
	\sgn(y-z)-\sgn(x-z).
	\]
	Note that
	\(
	|\eta_{x,y}(z)|
	\leq
	2\cdot\mathbf 1_{\{x\wedge y\leq z\leq x\vee y\}}
	\)
	and 
	\[
	D_t^y-D_t^x
	=
	X_t(\eta_{x,y})-M_t(\eta_{x,y}).
	\]
	For the finite-variation term,
	\[
	\mathbb E_{\delta_0}|X_t(\eta_{x,y})|
	\leq 2\int_{x\wedge y}^{x\vee y}p_t(z)\,\d z
	\longrightarrow0 \qquad\text{as }y\to x.
	\]
	For the martingale term, Lemma 3.1 of \cite{LeM05} yields
	\begin{align*}
	\mathbb E_{\delta_0}\!\left[\sup_{0\leq s\leq t}|M_s(\eta_{x,y})|^{1+\beta/2}\right]
	&\leq C\left[\int_0^t\!\d s\int_{I_{x,y}}|\eta_{x,y}(z)|^{(3+\beta)/2}p_s(z)\,\d z\right]^{(2+\beta)/(3+\beta)}\\
	&\quad+C\int_0^t\!\d s\int_{I_{x,y}}|\eta_{x,y}(z)|^{1+\beta/2} p_s(z)\,\d z,
	\end{align*}
	where $I_{x,y}=[x\wedge y,x\vee y]$. Since $|\eta_{x,y}|\leq2\mathbf1_{I_{x,y}}$ and $|I_{x,y}|\to0$, the right-hand side tends to zero. H\"older's inequality therefore gives $\mathbb E_{\delta_0}|M_t(\eta_{x,y})|\to0$. Hence
	\[
	\mathbb E_{\delta_0}|D_t^y-D_t^x|\longrightarrow0 \qquad\text{as }y\to x.
	\]
	Consequently,
	\[
	I_2(h)
	\leq
	\frac1{|h|}
	\int_{x\wedge(x+h)}^{x\vee(x+h)}
	\mathbb E_{\delta_0}|D_t^y-D_t^x|\,\d y
	\to 0 \ \text{as}\ h\to 0.
	\]
	
	Hence \eqref{H increment estimate} holds.
	
	Fix $a>0$ and set
	\[
	F_a(x):=\mathbb E_{\delta_0}|H_t^x-H_t^a|,
	\qquad x>0.
	\]
	Since
	\[
	|F_a(x+h)-F_a(x)|
	\leq
	\mathbb E_{\delta_0}|H_t^{x+h}-H_t^x|,
	\]
	\eqref{H increment estimate} implies $F_a'(x)=0$ for every $x>0$. Hence $F_a$ is constant on $(0,\infty)$. Because of $F_a(a)=0$, it follows that
	\[
	\mathbb E_{\delta_0}|H_t^x-H_t^a|=F_a(x)=0,
	\qquad x>0.
	\]
	Taking a countable intersection over
	$x\in\mathbb Q\cap(0,\infty)$, and using the continuity of
	$x\mapsto H_t^x$, we obtain, with probability one,
	\[
	H_t^x=H_t^a,
	\qquad x>0.
	\]
	Letting $x\downarrow0$ gives $H_t^a=H_t^0=0$, hence $H_t^x=0$ for every $x>0$ a.s.
	
	The argument on $(-\infty,0)$ is identical. Together with $H_t^0=0$, this proves the result.
	\qed
	
\end{proof}

We recall the time-change representation from Lemma 2.15 of \cite{FMW10}. For $p\geq 1$, denote by $\mathcal{L}_{loc}^{p}$ the space of equivalence classes of measurable functions $f$ such that
$$
\int_0^T\d s\int_{\mathbb R}|f(s,y)|^p p_s(y)\d y<\infty,\quad T>0.
$$
\begin{lemma}[\cite{FMW10}]\label{JT}
	Suppose $p\in(1+\beta,2)$ and let $f\in\mathcal{L}_{loc}^{p}$ with $f\ge0$. There exists a spectrally positive $(1+\beta)$-stable process $\{S_{s}:s\ge0\}$ such that
	\begin{equation*}
		M_t(f):=\int_{(0,t]\times\mathbb{R}}f(s,y)M(\d s,\d y)=S_{A_t(f)}, \quad t\ge0,
	\end{equation*}
	where
	\begin{equation}\label{M A kernel upper}
		A_t(f):=\int_{0}^{t}ds\int_{\mathbb{R}}(f(s,y))^{1+\beta}X_{s}(\d y).
	\end{equation}
\end{lemma}

\begin{lemma}[\cite{FMW10}]\label{JTT}
    Suppose that $\{S_s:s\geq 0\}$ is a spectrally positive $\kappa$-stable process, where $\kappa\in(1,2)$. Then there is a constant $ c_\kappa $ such that
    $$\mathbb{P}\left(\inf_{u\le t} S_u < -x\right) \le \exp\left\{ -c_\kappa \frac{x^{\kappa/(\kappa-1)}}{t^{1/(\kappa-1)}} \right\},\qquad x,t>0. $$
\end{lemma}

For $n\geq1$, set
\begin{equation}\label{Zn dyadic upper}
	Z_n:=[\mathsf R-2^{-n},\mathsf R]\cap(0,\infty).
\end{equation}

\begin{proposition}\label{theorem dyadic bootstrap upper}
	Assume that $0<\xi_0<1+2/\beta$ satisfies
	\begin{equation}\label{dyadic xi0 upper}
		\exists\,1\leq n_{\xi_0}(\omega)<\infty\ a.s.
		\text{ such that }
		L^x\leq 2^{-\xi_0 n},
		\qquad
		x\in Z_n,\ n\geq n_{\xi_0}(\omega).
	\end{equation}
	Then, for every
	\begin{equation}\label{xi improvement dyadic upper}
		0<\xi_1<
		1+\frac{1+\xi_0}{1+\beta},
	\end{equation}
	we have
	\begin{equation}\label{dyadic xi1 upper}
		\exists\,1\leq n_{\xi_1}(\omega)<\infty\ a.s.
		\text{ such that }
		L^x\leq 2^{-\xi_1 n},
		\qquad
		x\in Z_n,\ n\geq n_{\xi_1}(\omega).
	\end{equation}
\end{proposition}

We first deduce Theorem \ref{upper Holder continuity}.

\textbf{Proof of Theorem \ref{upper Holder continuity}.}
    Choose $\xi_0\in(0,\beta/(1+\beta))$. By Theorem
    \ref{theorem Lx estimation}, the map $x\mapsto L^x$ satisfies the pointwise $\eta$-H\"older condition at $\mathsf R$ for every
    $\eta<\beta/(1+\beta)$. Choose $\eta$ such that
    $\xi_0<\eta<\beta/(1+\beta)$.

    Since $L^{\mathsf R}=0$, the weak H\"older estimate with constant one implies
    that, for all sufficiently large $n$ and all $x\in Z_n$,
    \[
    L^x=|L^x-L^{\mathsf R}|\leq|\mathsf R-x|^\eta\leq2^{-\eta n}\leq2^{-\xi_0 n}.
    \]
    Thus \eqref{dyadic xi0 upper} holds.

    Let $0<\gamma<1+2/\beta$. Define
    \[
    F(\xi):=1+\frac{1+\xi}{1+\beta}.
    \]
    Then $1+2/\beta$ is the fixed point of $F$, and
    \(
    F(\xi)>\xi,
    \text{ for all } 0<\xi<1+2/\beta.
    \)
    Set $\xi^{(0)}=\xi_0$ and define inductively
    \[
    \xi^{(j+1)}
    :=
    \frac12\left(\xi^{(j)}+F(\xi^{(j)})\right).
    \]
    Then, as long as $\xi^{(j)}<1+2/\beta$, we have
    \(
    \xi^{(j)}<\xi^{(j+1)}<F(\xi^{(j)}).
    \)
     Moreover, $\xi^{(j)}\uparrow 1+2/\beta$. Hence, since $\gamma<1+2/\beta$, there exists
    $J<\infty$ such that
    \(
    \xi^{(J)}>\gamma .
    \)

    Applying Proposition \ref{theorem dyadic bootstrap upper}, we know
    \[
    \exists\,n_\gamma(\omega)<\infty
    \text{ such that }
    L^x\leq 2^{-\xi^{(J)}n},
    \qquad x\in Z_n,\ n\geq n_\gamma(\omega).
    \]

    Now let $0<\mathsf R-x<2^{-n_\gamma}$ and choose $n\geq n_\gamma$ such that
    \[
    2^{-(n+1)}<\mathsf R-x\leq 2^{-n}.
    \]
    Then $x\in Z_n$, and therefore
    \(
    L^x
    \leq
    2^{-\xi^{(J)}n}
    \leq
    2^{-\gamma n}.
    \)
    Since $2^{-(n+1)}<\mathsf R-x$, we have
    \(
    2^{-n}\leq 2(\mathsf R-x).
    \)

    Since $L^{\mathsf R}=0$, we obtain
    \[
    |L^x-L^{\mathsf R}|=L^x
    \leq
    2^\gamma|\mathsf R-x|^\gamma
    \]
    for all $x<\mathsf R$ sufficiently close to $\mathsf R$. If $x\geq\mathsf R$,
    then $L^x=L^{\mathsf R}=0$. Hence $L$ satisfies the pointwise $\gamma$-H\"older condition
    at the right boundary. The proof for the left boundary is identical by symmetry.
\qed

\textbf{Proof of Proposition \ref{theorem dyadic bootstrap upper}.}
    Fix $\xi_1$ satisfying \eqref{xi improvement dyadic upper}.

    Choose $\eta\in(0,\beta/(1+\beta))$. Apply Theorem
    \ref{theorem Lx estimation} on each deterministic interval $[-j,j]$, $j\in\mathbb N$, and take the countable intersection of the resulting probability-one events. Together with Theorem \ref{random interval}, this gives a random radius $\rho_0>0$ such that
    \begin{equation}\label{weak Holder dyadic upper}
	   |L^x-L^y|\leq |x-y|^\eta,
	   \qquad
	   x,y\in(\mathsf R-\rho_0,\mathsf R+\rho_0).
    \end{equation}

    For $m\geq2$, set
    \[
    \Omega_m:=
    \left\{
    \mathsf R\leq m,\ 
    \zeta\leq m,\ 
    \rho_0\geq m^{-1},\
    n_{\xi_0}\leq m
    \right\}.
    \]
    The events $\Omega_m$ are increasing and $\mathbb P_{\delta_0}(\Omega_m)\to1$. Fix $m\ge2$.

    On $\{\mathsf R\leq m,\ \zeta\leq m\}$, we have $L_m^r=0$ for every
    $r\in(m,m+1)$ a.s. Proposition \ref{Tanaka integral formula} therefore
    implies that $D_m^r=0$ for Lebesgue-a.e. $r\in(m,m+1)$ a.s. More precisely,
    \[
    \mathbb E_{\delta_0}\left[
    \mathbf1_{\{\mathsf R\leq m,\,\zeta\leq m\}}
    \int_m^{m+1}(|D_m^r|\wedge1)\,\d r
    \right]=0.
    \]
    By Fubini's theorem, we may therefore choose a deterministic $r_m\in(m,m+1)$ such that
    \[
    D_m^{r_m}=0
    \qquad
    \text{a.s. on }\{\mathsf R\leq m,\ \zeta\leq m\}.
    \]

    Let
    $$
    E_n:=
    \left\{
    L^x>2^{-\xi_1n}
    \text{ for some }x\in Z_n
    \right\},
    $$
    and set
    $$
    k_n:=
    \left\lceil
    \left(1+\frac{\xi_1}{\eta}\right)n
    \right\rceil,
    \qquad
    \Gamma_{n,m}:=
    \left\{
    j2^{-k_n}:j\in\mathbb N,\ 0<j2^{-k_n}<r_m
    \right\}.
    $$

    Since $\{r_m:m\ge2\}\cup\bigcup_{m,n}\Gamma_{n,m}$ is countable, we work
    on a common probability-one event on which Propositions \ref{prop Tanaka 4.1}
    and \ref{Tanaka integral formula} hold at all points used below.  Lemma
    \ref{JT} is applied separately to each deterministic kernel $h_{y,r_m}$.

    For all sufficiently large $n\ge m$,
    \begin{equation}\label{probability grid reduction}
        E_n\cap\Omega_m
        \subseteq
        \bigcup_{y\in\Gamma_{n,m}}
        \left(
        \left\{
        L^y>2^{-\xi_1n-1},\ y\in Z_n
        \right\}
        \cap\Omega_m
        \right).
    \end{equation}
    
    Indeed, suppose that $E_n\cap\Omega_m$ occurs.
    Choose $q\in Z_n$ such that
    $
    L^q>2^{-\xi_1n},
    $
    and let
    $
    y:=2^{-k_n}\left\lceil2^{k_n}q\right\rceil
    $
    be the first grid point on or to the right of $q$. Then
    $
    0\leq y-q\leq2^{-k_n}.
    $
    Since
    $
    k_n\eta\geq(\eta+\xi_1)n,
    $
    we have, for all sufficiently large $n$,
    $
    2^{-k_n\eta}
    \leq
    2^{-(\eta+\xi_1)n}
    \leq
    2^{-\xi_1n-1}.
    $
    Moreover, since $q\in Z_n$ and $k_n\geq n$,
    $$
    \mathsf R-2^{-n}
    \leq q\leq y
    \leq q+2^{-k_n}
    \leq\mathsf R+2^{-n}.
    $$
    On $\Omega_m$, for $n\geq m$,
    $
    2^{-n}<m^{-1}\leq\rho_0.
    $
    Thus both $q$ and $y$ lie in the region where
    \eqref{weak Holder dyadic upper} applies. Consequently,
    $$
    \begin{aligned}
    L^y
    \geq
    L^q-|q-y|^\eta 
    >
    2^{-\xi_1n}-2^{-k_n\eta}
    \geq
    2^{-\xi_1n-1}.
    \end{aligned}
    $$
    Since $L^x=0$ for $x\geq\mathsf R$, this lower bound implies
    that $y<\mathsf R$. Furthermore, $y\geq q\in Z_n$, and hence
    $$
    \mathsf R-2^{-n}\leq y<\mathsf R,
    \qquad y>0.
    $$
    Therefore $y\in Z_n$. Finally, on $\Omega_m$,
    $
    0<y<\mathsf R\leq m<r_m.
    $
    Since $y$ is a grid point, $y\in\Gamma_{n,m}$, proving \eqref{probability grid reduction}.
    
    It follows that, for all sufficiently large $n\geq m$,
    \begin{equation}\label{probability grid reduction upper}
    	\mathbb P_{\delta_0}(E_n\cap\Omega_m)
    	\leq
    	\sum_{y\in\Gamma_{n,m}}
    	\mathbb P_{\delta_0}\left(
    	\left\{
    	L^y>2^{-\xi_1n-1},\ y\in Z_n
    	\right\}
    	\cap\Omega_m
    	\right).
    \end{equation}
    
    Fix $y\in\Gamma_{n,m}$ and define
    \[
    F_{n,m}^y
    :=
    \left\{
    L^y>2^{-\xi_1n-1},\ y\in Z_n
    \right\}\cap\Omega_m.
    \]
    On $F_{n,m}^y$, by Proposition \ref{prop Tanaka 4.1} and
    \eqref{Dtx formula},
    \begin{align}\label{LLD}
    	L_m^y-L_m^{r_m}-(y-r_m)D_m^{r_m}
    	&=
    	X_m\bigl(g_y-g_{r_m}-(y-r_m)\sgn(r_m-\cdot)\bigr) \notag \\
    	&\quad-
    	M_m\bigl(g_y-g_{r_m}-(y-r_m)\sgn(r_m-\cdot)\bigr)  \notag \\
    	&=
    	X_m(h_{y,r_m})-M_m(h_{y,r_m}),
    \end{align}
    where we define
    \begin{equation}\label{h xr upper}
    \begin{aligned}
    	h_{x,r}(z)
    	&:=g_x(z)-g_r(z)-(x-r)\sgn(r-z) \\
    	&=2(z-x)\mathbf 1_{(x,r)}(z)+(r-x)\mathbf 1_{\{r\}}(z), \qquad \text{ for \(x<r\).}
    \end{aligned}
    \end{equation}
    Since \(L^y>0\),
    we have \(y<\mathsf R\), while the definition of \(\Omega_m\) gives
    \(\mathsf R\leq m\) and \(\zeta\leq m\).
    Therefore,
    \[
    L_m^y=L^y,\qquad
    X_m=0,\qquad
    L_m^{r_m}=0,\qquad
    D_m^{r_m}=0,
    \]
    and hence \eqref{LLD} reduces to
    \begin{equation}\label{localized Tanaka upper}
    	L^y=-M_m(h_{y,r_m}).
    \end{equation}
    
    The singleton term in \eqref{h xr upper} has zero Lebesgue contribution. Hence the occupation density formula gives
    \[
    \begin{aligned}
    A_m(h_{y,r_m})
    &=
    \int_{\mathbb R}
    h_{y,r_m}(z)^{1+\beta}L_m^z\,\d z 
    =
    2^{1+\beta}\int_y^{\mathsf R}
    (z-y)^{1+\beta}L^z\,\d z.
    \end{aligned}
    \]
    Moreover, 
    \eqref{dyadic xi0 upper} yields
    \begin{equation}\label{Ahy}
    A_m(h_{y,r_m})
    \leq C2^{-\xi_0n}(\mathsf R-y)^{\beta+2}
    \leq C2^{-n(\beta+2+\xi_0)}.
    \end{equation}

    By Lemma \ref{JT}, there exists a spectrally positive
    \((1+\beta)\)-stable process \(S^{y,m}\) such that
    \[
    M_m(h_{y,r_m})=S^{y,m}_{A_m(h_{y,r_m})}.
    \]
    Hence
    \begin{samepage}
    \begin{align}\label{PMA}
    &\left\{
    M_m(h_{y,r_m})<-2^{-\xi_1n-1},\
    A_m(h_{y,r_m})
    \leq C2^{-n(\beta+2+\xi_0)}
    \right\}  \notag\\
    &\subseteq \left\{
    S^{y,m}_{A_m(h_{y,r_m})}<-2^{-\xi_1n-1},\
    A_m(h_{y,r_m})
    \leq C2^{-n(\beta+2+\xi_0)}
    \right\}  \notag \\
    &\subseteq
    \left\{
    \inf_{0\leq u\leq C2^{-n(\beta+2+\xi_0)}}S^{y,m}_u
    <-2^{-\xi_1n-1}
    \right\}.
    \end{align}
    \end{samepage}

    Set
    \(
    \theta:=\beta+2+\xi_0-(1+\beta)\xi_1>0.
    \)
    Applying Lemma \ref{JTT} with $\kappa=1+\beta$, we obtain
    \begin{align*}
    &\mathbb P_{\delta_0}\left(
    \inf_{0\leq u\leq C2^{-n(\beta+2+\xi_0)}}S^{y,m}_u
    <-2^{-\xi_1n-1}
    \right) \\
    &\qquad\leq
    \exp\left\{
    -\bar c_\beta
    \frac{
    \bigl(2^{-\xi_1n-1}\bigr)^{(1+\beta)/\beta}
    }{
    \bigl(C2^{-n(\beta+2+\xi_0)}\bigr)^{1/\beta}
    }
    \right\}=
    \exp\left\{
    -c_1 2^{\theta n/\beta}
    \right\}. 
    \end{align*}
    where \(c_1:=\bar c_\beta C^{-1/\beta}2^{-(1+\beta)/\beta}>0.\)
    Hence \eqref{localized Tanaka upper} and \eqref{Ahy} give
    \begin{align}\label{PL}
    &\mathbb P_{\delta_0}\left(
    \left\{
    L^y>2^{-\xi_1n-1},\ y\in Z_n
    \right\}\cap\Omega_m
    \right) \notag \\
    &\quad\leq
    \mathbb P_{\delta_0}\left(
    M_m(h_{y,r_m})<-2^{-\xi_1n-1},\
    A_m(h_{y,r_m})
    \leq C2^{-n(\beta+2+\xi_0)}
    \right)  \notag \\
    &\quad\leq
    \exp\left\{
    -c_1 2^{\theta n/\beta}
    \right\}.
    \end{align}

    Combining \eqref{probability grid reduction upper} and \eqref{PL},
    \[
    \mathbb P_{\delta_0}(E_n\cap\Omega_m)
    \leq
    \#\Gamma_{n,m}
    \exp\left\{
    -c_1 2^{\theta n/\beta}
    \right\}.
    \]
    Since  \(
    \#\Gamma_{n,m}\leq(m+1)2^{k_n}+1
    \) and $k_n$ grows linearly in $n$, it follows that
    \(
    \sum_{n=1}^{\infty}
    \mathbb P_{\delta_0}(E_n\cap\Omega_m)<\infty.
    \)

    For every $m\geq2$, the Borel--Cantelli lemma gives
    \(
    \mathbb P_{\delta_0}
    \left(
    \Omega_m\cap\limsup_{n\to\infty}E_n
    \right)=0.
    \)
    Moreover, in view of the fact that $(\Omega_m)_{m\geq2}$ is increasing and
    $\mathbb P_{\delta_0}(\Omega_m)\to1$, we have
    \[
    \mathbb P_{\delta_0}
    \left(\limsup_{n\to\infty}E_n\right)=0.
    \]
    Hence the events $E_n$ occur only finitely often almost surely,
    which means that there exists an almost surely finite random integer
    $n_{\xi_1}$ such that for every $x\in Z_n, n\geq n_{\xi_1}$,
    \(
    L^x\leq2^{-\xi_1n}.
    \)
\qed

\section{Lower bound of the local time near the boundary}
\setcounter{equation}{0}

We prove the matching lower boundary estimate.

Recall the spatial CSBP $Y$ and its extinction point $\mathsf R$. For $n\geq1$, set
\[
\tau_n:=\inf\{r\geq0:Y_r\leq2^{-n}\}.
\]
Since $Y$ is absorbed at zero and has no negative jumps,
\begin{equation}\label{tau n properties}
\tau_n<\infty,\qquad \tau_n\uparrow\mathsf R,\qquad Y_{\tau_n}=2^{-n}
\quad\mathbb P_{\delta_0}\text{-a.s.}
\end{equation}
Moreover, the stopping times used in Section 3 satisfy $R_n=n\wedge\tau_n$. Since $\mathsf R<\infty$ almost surely, $R_n=\tau_n$ for all sufficiently large $n$.

\begin{lemma}\label{lem Rn approximation}
 Fix $0<\xi<\beta/2$. Then, $\mathbb P_{\delta_0}$-almost surely, there exists a finite random integer $N_\xi$ such that
 \[
 \mathsf R-\tau_n\leq2^{-n\xi},\qquad n\geq N_\xi.
 \]
\end{lemma}
\begin{proof}
 By \eqref{tau n properties} and the strong Markov property of $Y$ at $\tau_n$,
 \begin{align*}
 \mathbb P_{\delta_0}(\mathsf R-\tau_n>2^{-n\xi})
 &=\mathbb P_{\delta_0}(Y_{\tau_n+2^{-n\xi}}>0)\\
 &=\mathbb E_{\delta_0}\left[
 \mathbb P_{Y_{\tau_n}\delta_0}(Y_{2^{-n\xi}}>0)
 \right]\\
 &=\mathbb P_{2^{-n}\delta_0}(Y_{2^{-n\xi}}>0).
 \end{align*}
 By \eqref{Yr},
 \[
 \mathbb P_{z\delta_0}(Y_t>0)
 =1-\exp\left\{-z\left(\frac{\beta t}{\sqrt{\beta+2}}\right)^{-2/\beta}\right\},
 \qquad z,t>0.
 \]
 Hence
 \begin{align*}
 \mathbb P_{\delta_0}(\mathsf R-\tau_n>2^{-n\xi})
 &\leq
 \left(\frac{\beta}{\sqrt{\beta+2}}\right)^{-2/\beta}
 2^{-n(1-2\xi/\beta)}.
 \end{align*}
 Since $\xi<\beta/2$, the right-hand side is summable in $n$, and the Borel--Cantelli lemma completes the proof. \qed
\end{proof}

\noindent\textbf{Proof of Theorem \ref{lower Holder continuity}.}
    Fix $\gamma>1+2/\beta$, and choose $\xi$ and $H$ such that
    \(
    \frac{1+\beta/2}{\gamma}<\xi<\frac{\beta}{2},
    \) and 
    \(
    1+\frac{\beta}{2}<H<\gamma\xi.
    \)
    It suffices to prove
    \begin{equation}\label{Llower}
    	\lim_{x\uparrow\mathsf R}
    	\frac{|L^x-L^{\mathsf R}|}{(\mathsf R-x)^\gamma}
    	=\infty,
    \end{equation}

    In the proof of Theorem \ref{random interval}, the lower estimate was obtained for every exponent of the form
    $H=(1+2/\beta)\eta$ with $\eta>\beta/2$. The present assumption
    $H>1+\beta/2$ is exactly equivalent to
    \[
      \eta:=\frac{H}{1+2/\beta}>\frac{\beta}{2}.
    \]
    Hence that argument applies to the present choice of $H$ and yields, almost
    surely, for all sufficiently large $n$,
    \begin{equation}\label{local time lower Holder}
    	\inf_{0\leq x<R_n}L^x\geq 2^{-nH}.
    \end{equation}
    Since $R_n=\tau_n$ eventually, the same estimate holds with $\tau_n$ in place of $R_n$ for all sufficiently large $n$.
    Moreover, Lemma \ref{lem Rn approximation} implies that, almost surely,
    \(
    \tau_n\uparrow\mathsf R
    \)
    and, for all sufficiently large $n$,
    \(
    \mathsf R-\tau_n\leq 2^{-n\xi}.
    \)

    Choose $N$ so that both bounds hold for every $n\ge N$. For this choice of $N$, we have
    \(
    [\tau_N,\mathsf R)
    =
    \bigcup_{n\geq N}[\tau_n,\tau_{n+1}).
    \)
    
    For $x\in[\tau_N,\mathsf R)$, choose $n\ge N$ such that
    \(
    x\in[\tau_n,\tau_{n+1})
    \)
    for a certain $n\geq N$. Applying \eqref{local time lower Holder} with
    $n+1$ in place of $n$ gives
    \[
    L^x\geq 2^{-(n+1)H}.
    \]
    Since $x\geq\tau_n$, Lemma \ref{lem Rn approximation} gives
    \[
    \mathsf R-x
    \leq\mathsf R-\tau_n
    \leq2^{-n\xi}.
    \]
    Since $L^{\mathsf R}=0$, we  now combine these two bounds to arrive at
    \begin{equation}\label{Holder quotient lower bound}
    	\frac{|L^x-L^{\mathsf R}|}{(\mathsf R-x)^\gamma}
    	\geq 
    	\frac{2^{-(n+1)H}}{2^{-n\gamma\xi}}
    	=
    	2^{-H}2^{n(\gamma\xi-H)}.
    \end{equation}

    As $x\uparrow\mathsf R$, the corresponding $n\to\infty$. Since $\gamma\xi-H>0$, \eqref{Holder quotient lower bound} yields \eqref{Llower}. Thus $L$ fails the pointwise $\gamma$-H\"older condition at $\mathsf R$. The assertion at $\mathsf L$ follows by symmetry.
\qed

\section{The case under the canonical measure}
\setcounter{equation}{0}

The Poisson cluster decomposition \eqref{X poisson point representation}--\eqref{L poisson point representation} transfers the preceding results to $\mathbb N_0$. It remains to prove positivity at the root of a canonical cluster.

\begin{lemma}\label{canonical positivity at zero}
 We have
 \[
  L^0>0,\qquad \mathbb N_0\text{-a.e.}
 \]
\end{lemma}

\begin{proof}
For $\lambda>0$ and $\theta>0$, set
\[
W_\lambda(x):=\mathbb N_x\!\left(\mathcal O(1)e^{-\lambda L^0}\right),
\qquad
U_{\lambda,\theta}(x):=\mathbb N_x\!\left(1-e^{-\lambda L^0-\theta\mathcal O(1)}\right),
\]
and
\[
D_{\lambda,\theta}(x):=\frac{U_{\lambda,\theta}(x)-V^\lambda(x)}{\theta}.
\]
Then
\[
D_{\lambda,\theta}(x)
=\mathbb N_x\!\left[e^{-\lambda L^0}
\frac{1-e^{-\theta\mathcal O(1)}}{\theta}\right]
\uparrow W_\lambda(x)
\quad(\theta\downarrow0)
\]
by monotone convergence.

To justify the singular source at the origin, let $r_\varepsilon$ be the
approximate identity from Lemma \ref{lemma V lambda} and set
\[
U^{\varepsilon}_{\lambda,\theta}(x)
:=\mathbb N_x\!\left(1-e^{-\lambda\mathcal O(r_\varepsilon)-\theta\mathcal O(1)}\right).
\]
The canonical log--Laplace equation gives
\begin{equation}\label{canonical mollified equation}
\frac12\Delta U^{\varepsilon}_{\lambda,\theta}
=(U^{\varepsilon}_{\lambda,\theta})^{1+\beta}-\lambda r_\varepsilon-\theta
\end{equation}
in the distributional sense.  By the Poisson representation and continuity of
$x\mapsto L^x$,
\[
e^{-U^{\varepsilon}_{\lambda,\theta}(x)}
=\mathbb E_{\delta_x}\!\left[e^{-\lambda\mathcal O(r_\varepsilon)-\theta\mathcal O(1)}\right]
\longrightarrow e^{-U_{\lambda,\theta}(x)}.
\]
Moreover, by the Poisson representation applied to the mollified
occupation functional,
\[
V^{\lambda,\varepsilon}(x)
=\mathbb N_x\!\left(1-e^{-\lambda\mathcal O(r_\varepsilon)}\right).
\]
Hence it follows that
\[
V^{\lambda,\varepsilon}(x)=\mathbb N_x\!\left(1-e^{-\lambda\mathcal O(r_\varepsilon)}\right)\leq  \mathbb N_x\!\left(1-e^{-\lambda\mathcal O(r_\varepsilon)-\theta\mathcal O(1)}\right)= U^{\varepsilon}_{\lambda,\theta}(x).
\] 
On the other hand, using
\[
1-e^{-(a+b)}\le (1-e^{-a})+(1-e^{-b}),\qquad a,b\ge0,
\]
we obtain
\[
 U^{\varepsilon}_{\lambda,\theta}(x)
\le V^{\lambda,\varepsilon}(x)+U_{0,\theta}(x).
\]
By spatial translation invariance, $U_{0,\theta}=U_{0,\theta}^{\varepsilon}$ is constant, and the
canonical log--Laplace equation gives
\[
U_{0,\theta}(x)=\mathbb N_x\!\left(1-e^{-\theta\mathcal O(1)}\right)
=\theta^{1/(1+\beta)}.
\]
Therefore we conclude that
\[
V^{\lambda,\varepsilon}(x)
\le U^{\varepsilon}_{\lambda,\theta}(x)
\le V^{\lambda,\varepsilon}(x)+\theta^{1/(1+\beta)}.
\]
Together with the uniform bound from Lemma \ref{lemma V lambda}, this permits
passing to the limit in \eqref{canonical mollified equation} on compact sets.
Thus
\begin{equation}\label{canonical singular equation}
\frac12\Delta U_{\lambda,\theta}
=U_{\lambda,\theta}^{1+\beta}-\lambda\delta_0-\theta
\end{equation}
in distributions, and
\begin{equation}\label{canonical comparison U}
V^\lambda\le U_{\lambda,\theta}\le V^\lambda+\theta^{1/(1+\beta)}.
\end{equation}

Subtracting \eqref{V lambda} from \eqref{canonical singular equation}, the
singular terms at the origin cancel.  Consequently
\begin{equation}\label{canonical secant equation}
\frac12D_{\lambda,\theta}''-q_{\lambda,\theta}D_{\lambda,\theta}=-1,
\end{equation}
where
\[
q_{\lambda,\theta}(x)
:=\frac{U_{\lambda,\theta}(x)^{1+\beta}-V^\lambda(x)^{1+\beta}}
{U_{\lambda,\theta}(x)-V^\lambda(x)},
\]
with the derivative value $(1+\beta)(V^\lambda(x))^\beta$ when the denominator
vanishes.  In particular the derivative jumps of $U_{\lambda,\theta}$ and
$V^\lambda$ at $0$ are both $-2\lambda$, so their difference has no singular
part.  Since $q_{\lambda,\theta}$ is continuous,
\eqref{canonical secant equation} implies $D_{\lambda,\theta}\in C^2(\mathbb R)$.

By \eqref{canonical comparison U},
\begin{equation}\label{D theta uniform bound}
0\le D_{\lambda,\theta}(x)\le\theta^{-\beta/(1+\beta)}.
\end{equation}
Also $U_{\lambda,\theta}\ge U_{0,\theta}=\theta^{1/(1+\beta)}$ by monotonicity.
Since, for $a\ge b\ge0$,
\[
a^{1+\beta}-b^{1+\beta}\ge a^\beta(a-b),
\]
we obtain
\begin{equation}\label{q theta lower bound}
q_{\lambda,\theta}(x)\ge\theta^{\beta/(1+\beta)}.
\end{equation}

Let $\tau_R=\inf\{t\ge0:|B_t|\ge R\}$.  The Dirichlet Feynman--Kac formula on
$(-R,R)$ \cite[Chapter~7]{Le16} gives, for $|x|<R$,
\begin{align*}
D_{\lambda,\theta}(x)
&=\mathbb E_x\!\left[\int_0^{\tau_R}
\exp\left\{-\int_0^tq_{\lambda,\theta}(B_s)\,\d s\right\}\d t\right]\\
&\quad+\mathbb E_x\!\left[
\exp\left\{-\int_0^{\tau_R}q_{\lambda,\theta}(B_s)\,\d s\right\}
D_{\lambda,\theta}(B_{\tau_R})\right].
\end{align*}
By \eqref{D theta uniform bound}--\eqref{q theta lower bound}, the second term
is bounded by
\[
\theta^{-\beta/(1+\beta)}
\mathbb E_x\!\left[e^{-\theta^{\beta/(1+\beta)}\tau_R}\right]
\longrightarrow0
\qquad(R\to\infty),
\]
because $\tau_R\uparrow\infty$ almost surely.  Therefore
\begin{equation}\label{D theta FK}
D_{\lambda,\theta}(x)
=\mathbb E_x\!\left[\int_0^\infty
\exp\left\{-\int_0^tq_{\lambda,\theta}(B_s)\,\d s\right\}\d t\right].
\end{equation}

By \eqref{canonical comparison U},
$U_{\lambda,\theta}\downarrow V^\lambda$ as $\theta\downarrow0$.  Since
$a\mapsto(a^{1+\beta}-b^{1+\beta})/(a-b)$ is increasing for $a> b$,
\[
q_{\lambda,\theta}(x)\downarrow(1+\beta)(V^\lambda(x))^\beta.
\]
Monotone convergence in \eqref{D theta FK} yields
\begin{equation}\label{canonical first variation}
W_\lambda(0)
=\mathbb E_0\!\left[\int_0^\infty
\exp\left\{-\int_0^t(1+\beta)(V^\lambda(B_s))^\beta\,\d s\right\}\d t\right].
\end{equation}

By \eqref{precise V lambda},
\[
(1+\beta)(V^\lambda(x))^\beta
=\frac{\kappa_\beta}{(|x|+a_\lambda)^2},
\qquad
\kappa_\beta=\frac{(1+\beta)(\beta+2)}{\beta^2},
\]
where
\[
a_\lambda
=\frac{\sqrt{\beta+2}}{\beta}
\left(\frac{2}{\lambda\sqrt{\beta+2}}\right)^{\beta/(\beta+2)}.
\]
Brownian scaling gives
\[
W_\lambda(0)=a_\lambda^2C_\beta,
\qquad
C_\beta:=\mathbb E_0\!\left[\int_0^\infty
\exp\left\{-\kappa_\beta\int_0^t\frac{\d s}{(1+|B_s|)^2}\right\}\d t\right].
\]
Set
\[
q(x):=\frac{\kappa_\beta}{(1+|x|)^2},
\qquad
A_t:=\int_0^tq(B_s)\,\d s,
\qquad
h(x):=\frac{1+x^2}{\kappa_\beta/2-1}.
\]
Since $\kappa_\beta>2$ and
$\frac{1+x^2}{(1+|x|)^2}\ge\frac12$,
\[
\frac12h''(x)-q(x)h(x)
=\frac{1}{\kappa_\beta/2-1}
\left(1-\kappa_\beta\frac{1+x^2}{(1+|x|)^2}\right)
\le-1.
\]
Let
\[
\tau_n:=\inf\{t\ge0:|B_t|\ge n\},
\qquad
\sigma_n:=\tau_n\wedge n.
\]
On $[0,\sigma_n]$, the stochastic integral in It\^o's formula for
$e^{-A_t}h(B_t)$ is a true martingale.  Hence
\[
\mathbb E_0[e^{-A_{\sigma_n}}h(B_{\sigma_n})]
=h(0)+\mathbb E_0\!\left[\int_0^{\sigma_n}e^{-A_s}
\left(\frac12h''(B_s)-q(B_s)h(B_s)\right)\d s\right],
\]
and therefore
\begin{equation}\label{stopped FK bound}
\mathbb E_0\!\left[\int_0^{\sigma_n}e^{-A_s}\,\d s\right]\le h(0).
\end{equation}
Since $\sigma_n\uparrow\infty$ a.s., monotone convergence gives
\[
C_\beta\le h(0)=\frac{1}{\kappa_\beta/2-1}<\infty.
\]
As $\lambda\to\infty$, $a_\lambda\to0$, and hence
\[
W_\lambda(0)=\mathbb N_0(\mathcal O(1)e^{-\lambda L^0})\longrightarrow0.
\]
Fix $\lambda_0>0$.  Since
$\mathcal O(1)e^{-\lambda_0L^0}$ is $\mathbb N_0$-integrable, dominated
convergence, applied for $\lambda\ge\lambda_0$, yields
\[
\mathbb N_0\!\left(\mathcal O(1)\mathbf1_{\{L^0=0\}}\right)=0.
\]
By \eqref{canonical occupation identity}, $\mathcal O(1)=\sigma>0$,
$\mathbb N_0$-a.e.; hence $\mathbb N_0(L^0=0)=0$. \qed
\end{proof}

\noindent\textbf{Proof of Theorem \ref{canonical random interval}.}
For a canonical cluster, set
\[
 \mathsf R=\sup\{x\geq0:L^x>0\},
 \qquad
 \mathsf L=\inf\{x\leq0:L^x>0\}.
\]
We use the continuous version of the occupation density under $\mathbb N_0$.
Lemma \ref{canonical positivity at zero} and continuity show that
$\mathsf L<0<\mathsf R$, $\mathbb N_0$-a.e.

The endpoints are finite. Indeed, if $Y_r\delta_r$ is the exit measure from
$(-\infty,r)$, then the canonical version of \eqref{Yr} is
\[
 \mathbb N_0(1-e^{-\lambda Y_r})=U^{\lambda,r}(0).
\]
Letting $\lambda\uparrow\infty$ and using \eqref{U lambda r x}, we get
\begin{equation}\label{canonical halfline exit intensity}
 \mathbb N_0(Y_r>0)
 =\left(\frac{\beta r}{\sqrt{\beta+2}}\right)^{-2/\beta}<\infty.
\end{equation}
If $\mathsf R>r$, some genealogical spatial path reaches a point larger than $r$ and, by continuity, exits $(-\infty,r)$; hence $Y_r>0$. Thus
$\mathbb N_0(\mathsf R=\infty)=0$ by letting $r\uparrow\infty$ in
\eqref{canonical halfline exit intensity}. Reflection gives
$\mathbb N_0(\mathsf L=-\infty)=0$.

It remains to rule out zeros in the interior. Fix $\varepsilon>0$ and write
\[
 m_\varepsilon=\mathbb N_0(\mathsf R>\varepsilon).
\]
Then $0<m_\varepsilon<\infty$. Finiteness follows from
\eqref{canonical halfline exit intensity}, while positivity follows, for example,
from
\[
 \mathbb N_0(L^{2\varepsilon}>0)
 =\lim_{\lambda\uparrow\infty}V^\lambda(2\varepsilon)>0.
\]
Let $\Xi=\sum_{i\in I}\delta_{W_i}$ be the Poisson point measure with intensity
$\mathbb N_0$ which produces a superprocess with law $\mathbb P_{\delta_0}$, and let
\[
 N_\varepsilon
 =\#\{i\in I:\mathsf R(W_i)>\varepsilon\}.
\]
Thus $N_\varepsilon$ is Poisson with mean $m_\varepsilon$. If
$A\subset\{\mathsf R>\varepsilon\}$ is measurable, the elementary Poisson
identity
\begin{equation}\label{one canonical cluster identity}
 \mathbb P_{\delta_0}\left(
  N_\varepsilon=1;
  \text{the unique selected cluster belongs to }A
 \right)
 =e^{-m_\varepsilon}\mathbb N_0(A)
\end{equation}
will be used repeatedly.

Take
\[
 A_\varepsilon=
 \{\mathsf R>\varepsilon:\text{ there exists }x\in(\varepsilon,\mathsf R)
   \text{ such that }L^x=0\}.
\]
On the event in the left-hand side of
\eqref{one canonical cluster identity} with $A=A_\varepsilon$, all other
clusters vanish on $(\varepsilon,\infty)$. Hence the total local time under
$\mathbb P_{\delta_0}$ has a zero strictly between $\varepsilon$ and its right
endpoint. Theorem \ref{random interval} shows that this event has probability
zero. It follows from \eqref{one canonical cluster identity} that
$\mathbb N_0(A_\varepsilon)=0$.

Letting $\varepsilon$ decrease to zero through the positive rationals, and using
$\mathsf R>0$, we conclude that
\[
 L^x>0,\qquad 0<x<\mathsf R,
 \qquad \mathbb N_0\text{-a.e.}
\]
By reflection,
$L^x>0$ for every $\mathsf L<x<0$, $\mathbb N_0$-a.e. Together with
Lemma \ref{canonical positivity at zero} and the definitions of the endpoints,
this proves
\[
 \{x:L^x>0\}=(\mathsf L,\mathsf R),
 \qquad \mathbb N_0\text{-a.e.} \tag*{$\square$}
\]

\noindent\textbf{Proof of Theorem \ref{canonical Holder continuity}.}
Fix $0<\gamma<1+2/\beta$ and consider the right endpoint. With $N_\varepsilon$ as above, let
\[
 A^{\mathrm{up}}_{\varepsilon,\gamma}
 =\{\mathsf R>\varepsilon:
   L\text{ does not satisfy the pointwise }\gamma\text{-H\"older condition at }\mathsf R\}.
\]
If $N_\varepsilon=1$ and the unique selected cluster belongs to
$A^{\mathrm{up}}_{\varepsilon,\gamma}$, then on a neighborhood of the common
right endpoint the total local time is exactly the local time of that cluster.
This would contradict Theorem \ref{upper Holder continuity}. Therefore
\eqref{one canonical cluster identity} gives
\[
 \mathbb N_0(A^{\mathrm{up}}_{\varepsilon,\gamma})=0.
\]
Letting $\varepsilon\downarrow0$ through the rationals proves the upper assertion
at $\mathsf R$.

For $\gamma>1+2/\beta$, set
\[
 A^{\mathrm{low}}_{\varepsilon,\gamma}
 =\{\mathsf R>\varepsilon:
   L\text{ satisfies the pointwise }\gamma\text{-H\"older condition at }\mathsf R\}.
\]
The same one-cluster argument, this time using Theorem
\ref{lower Holder continuity}, yields
$\mathbb N_0(A^{\mathrm{low}}_{\varepsilon,\gamma})=0$ for every
$\varepsilon>0$. Letting $\varepsilon\downarrow0$ proves the lower assertion at
$\mathsf R$. Reflection gives the corresponding assertions at $\mathsf L$. Taking countable
intersections over rational exponents yields the claims for all exponents in the
stated ranges.
\qed

\noindent\textbf{Proof of Corollary \ref{canonical stable Levy measure}.}
Taking $x=x_0=0$ in the canonical Laplace identity obtained in
\eqref{general Laplace of Lx}--\eqref{precise V lambda}, we obtain
\[
 \mathbb N_0(1-e^{-\lambda L^0})
 =V^\lambda(0)
 =\left(\frac{\lambda\sqrt{\beta+2}}{2}\right)^{\frac{2}{2+\beta}}
 =c_\beta\lambda^\alpha.
\]
For $0<\alpha<1$,
\[
 \int_0^\infty(1-e^{-\lambda\ell})\ell^{-1-\alpha}\d\ell
 =\frac{\Gamma(1-\alpha)}{\alpha}\lambda^\alpha.
\]
Together with Lemma \ref{canonical positivity at zero}, we may apply the uniqueness of Laplace transforms for L\'evy measures to get \eqref{e3.7}.
\qed

%


\end{document}